\documentclass[english,notitlepage]{amsart}
\usepackage[latin9]{inputenc}
\usepackage{geometry}
\usepackage{amsthm}
\usepackage{amsmath, amscd}
\usepackage{amssymb}
\usepackage{mathrsfs}
\usepackage{esint}

\usepackage{color}

\makeatletter
\numberwithin{equation}{section}
\numberwithin{figure}{section}

\theoremstyle{plain}
\newtheorem{thm}{Theorem}
  \theoremstyle{plain}
  \numberwithin{thm}{section}
  \newtheorem{cor}[thm]{Corollary}
  \theoremstyle{plain}
  \newtheorem{lem}[thm]{Lemma}
  \theoremstyle{plain}
   \newtheorem{proposition}[thm]{Proposition}
    
  \theoremstyle{remark}
  \newtheorem{rem}[thm]{Remark}

  \def\Ddots{\mathinner{\mkern1mu\raise\p@
\vbox{\kern7\p@\hbox{.}}\mkern2mu
\raise4\p@\hbox{.}\mkern2mu\raise7\p@\hbox{.}\mkern1mu}}
\makeatother

\usepackage{babel}

\newcommand{\eps}{\varepsilon}

\newcommand{\norm}[1]{\left\| #1 \right\|}
\newcommand{\mklm}[1]{\left\{ #1 \right\}}
\newcommand{\eklm}[1]{\left\langle #1 \right\rangle}

\renewcommand{\d}{\,d}
\newcommand{\N}{{\mathbb N}}
\newcommand{\Z}{{\mathbb Z}}
\newcommand{\C}{{\mathbb C}}

\newcommand{\R}{{\mathbb R}}
\newcommand{\A}{{\mathcal A}}

\newcommand{\F}{{\mathcal F}}

\renewcommand{\H}{{\mathcal H}}

\newcommand{\M}{{\mathcal M}}

\newcommand{\Q}{{\mathbb Q}}
\renewcommand{\O}{{\mathcal O}}

\newcommand{\1}{{\bf 1}}
\renewcommand{\epsilon}{\varepsilon}

\renewcommand{\rho}{\varrho}

\newcommand{\Cinft}{{\rm C^{\infty}}}
\newcommand{\CT}{{\rm C^{\infty}_c}}

\renewcommand{\L}{{\rm L}}
\newcommand{\Lcal}{{\mathcal L}}

\renewcommand{\S}{{\mathcal S}}

\newcommand{\GL}{\mathrm{GL}}
\newcommand{\PGL}{\mathrm{PGL}}
\newcommand{\SL}{\mathrm{SL}}
\newcommand{\SO}{\mathrm{SO}}
\newcommand{\PO}{\mathrm{PO}}
\newcommand{\SU}{\mathrm{SU}}
\newcommand{\Sp}{\mathrm{Sp}}

\newcommand{\g}{{\bf \mathfrak g}}
\renewcommand{\k}{{\bf \mathfrak k}}
\renewcommand{\t}{{\bf \mathfrak t}}
\newcommand{\p}{{\bf \mathfrak p}}
\newcommand{\fa}{{\bf \mathfrak a}}

\newcommand{\fb}{{\bf \mathfrak b}}

\newcommand{\Ad}{\mathrm{Ad}\,}
\newcommand{\ad}{\mathrm{ad}\,}

\renewcommand{\det}{\mathrm{det}\,}

\newcommand{\vol}{\mathrm{vol}\,}
\newcommand{\dist}{\mathrm{dist}\,}

\DeclareMathOperator{\supp}{supp\,}

\DeclareMathOperator{\tr}{tr}
\DeclareMathOperator{\gd}{\partial}

\DeclareMathOperator{\Tr}{Tr}

\newcommand{\bdm}{\begin{displaymath}}
\newcommand{\edm}{\end{displaymath}}
\newcommand{\bq}{\begin{equation}}
\newcommand{\eq}{\end{equation}}
\newcommand{\bqn}{\begin{equation*}}
\newcommand{\eqn}{\end{equation*}}

\newcommand{\bsl}{\backslash}
\newcommand{\cR}{\mathcal R}

\newcommand{\bH}{\mathbb H}
\newcommand{\bA}{\mathbb A}

\newcommand{\fin}{\mathrm{fin}}

\newcommand{\Gal}{\mathrm{Gal}}

\begin{document}
\author{Pablo Ramacher and Satoshi Wakatsuki}
\title[Asymptotics for  Hecke eigenvalues on compact arithmetic quotients]{Asymptotics for Hecke eigenvalues of automorphic forms on compact arithmetic quotients} 
\address{Fachbereich 12 Mathematik und Informatik, Philipps-Universit\"at Marburg, Hans--Meerwein-Str. 6, 35043 Marburg, Germany}
\email{ramacher@mathematik.uni-marburg.de}
\address{Faculty of Mathematics and Physics, Institute of Science and Engineering, Kanazawa University, Kakumamachi, Kanazawa, Ishikawa, 920-1192, Japan}
\email{wakatsuk@staff.kanazawa-u.ac.jp}

\date{November 20, 2020}

\begin{abstract}
In this paper, we describe the asymptotic distribution of Hecke eigenvalues in the Laplace eigenvalue aspect for certain families of Hecke-Maass forms on compact arithmetic quotients. Instead of relying on the trace formula, which was the primary tool in preceding studies on the subject, we use Fourier integral operator methods. This allows us to treat not only spherical, but also non-spherical Hecke-Maass forms with corresponding remainder estimates.
Our asymptotic formulas are available for arbitrary simple and connected algebraic groups over number fields with cocompact arithmetic subgroups. 
\end{abstract}

\maketitle

\setcounter{tocdepth}{1}
\tableofcontents{}

\section{Introduction}\label{sec:1}

In this paper, we describe the asymptotic distribution  of Hecke eigenvalues for automorphic forms on compact arithmetic quotients of semisimple Lie groups in the Laplace eigenvalue aspect in both spherical and non-sperical settings.  In the spherical case, such asymptotics were studied previously by Sarnak  \cite{Sarnak1984} for $\SL(2,\R)$, Imamoglu-Raulf  \cite{IR} for $\SL(2,\C)$,  Matz  \cite{Matz} for  $\SL(n,\C)$, and Matz-Templier  \cite{MT} for $\SL(n,\R)$,  as well as Finis-Matz \cite{FM} and Finis-Lapid \cite{FL} for all simple reductive groups. To our knowledge, non-spherical asymptotics have not been considered so far. Asymptotic formulas for Hecke eigenvalues are related to the asymptotic distribution of Laplace eigenvalues of Hecke-Maas forms, as described by  Weyl's law \cite{DKV1, Mueller, Lapid-Mueller, ramacher10}, and imply   Plancherel density theorems and Sato-Tate equidistribution theorems for Hecke eigenvalues \cite{Sarnak1984, CDF, Serre, Shin, ST, BBR}. From the latter, statistics of low-lying zeros of automorphic $L$-functions can be inferred  \cite{MT, ST}. 

Our results cover basically two settings.  Let $H$ denote a semisimple connected linear algebraic group over the rational number field $\Q$, and write $\bA$ for the adele  ring of $\Q$. As usual, regard $H(\Q)$ as a subgroup of $H(\bA)$ by the diagonal embedding, and assume that $H(\Q)\bsl H(\bA)$ is compact. On the one hand, we derive non-equivariant asymptotics with remainder for Hecke eigenvalues of  automorphic forms in the space $L^2(H(\Q)\bsl H(\bA))$ of square integrable functions on $H(\Q)\bsl H(\bA)$, see Theorem \ref{thm:main}.  On the other hand, we prove equivariant asymptotics with remainder for Hecke eigenvalues of automorphic forms belonging to specific $\sigma$-isotypic components $L^2_\sigma(H(\Q)\bsl H(\bA))$ in  $L^2(H(\Q)\bsl H(\bA))$ for any simple connected algebraic group $H$ and any $K$-type $\sigma \in \widehat K$, where $K$ is a maximal compact subgroup of $G:=H(\R)$, see Theorem \ref{thm:equiv}.  
Our asymptotic formulas for Hecke eigenvalues do imply corresponding Plancherel density theorems and  Sato-Tate equidistribution theorems, see Corollaries \ref{plancherel density} and \ref{Sato-Tate equiv} for the non-equivariant, as well as Corollaries \ref{cor:equiv.planch}  and \ref{cor:equiv.ST} for the equivariant case.  They can also be applied to the study of statistics of low-lying zeros of automorphic $L$-functions as iniciated by Katz-Sarnak in \cite{KS}. The heuristics developed by them suggests that  $L$-functions can be grouped into families  according to the symmetry type exhibited in the distribution of their zeros, and was refined and confirmed for larger classes in \cite{ST, SST}.  In fact, following the approach of \cite[Section 11]{ST} and under some additional assumptions, our results can be used to study low-lying zeros in  rather general situations, see Theorem \ref{thm:lowlyingzeros}.

The major novelty of our approach,  initiated  in \cite{RW}, consists in applying methods from the modern theory of partial differential equations, more precisely, the theory of Fourier integral operators, to the analysis of Hecke--Maass forms. This allows us  to circumvent the study of the geometric side  of the trace formula, the main tool of previous approaches, and study not only spherical Hecke--Maass forms, but also non-spherical ones, which were deemed difficult to treat up to now. In fact, the usual way to proceed in the spherical case is to test the trace formula with a spherical function, and then apply the  Fourier inversion formula for the derivation of  upper bounds \cite{MT,FM}.  A construction of non-spherical test functions based on Arthur's Paley-Wiener theorem was set up in \cite[Proposition 1]{CD84}. 
However, for non-trivial $K$-types the Fourier inversion formula  is unavailable in general \cite{MatzProc}, and so far was established only for one dimensional $K$-types  \cite{Patterson,Shimeno}. An alternative way could consist in combining Herb's explicit formula for the inverse Fourier transform of semisimple orbital integrals \cite{Herb} with  the invariant  Paley-Wiener theorem \cite{CD1990}, since in the cocompact case all $\Q$-rational points are semisimple. This approach requires the study of $K$-type multiplicities of unitary representations and Fourier inversion for singular orbital integrals. In contrast, our method avoids such difficulties.
Let us also mention that in the spherical situations considered previously, spectral asymptotics for  the whole algebra of invariant differential operators were derived, which leads to more refined statements in the  higher rank case, compare  \cite[Section 8]{DKV1}. Within the FIO approach, one only considers eigenfunctions of  a single elliptic differential operator, yielding less refined asymptotics. Nevertheless, they already suffice to derive Plancherel and Sato-Tate  theorems, yielding equidistribution results in a simpler way.  In a future paper, we intend to generalize our approach  to non-compact arithmetic quotients. 

The structure of this paper is as follows. In Section  \ref{sec:nonadelic} we explain our results within a non-adelic framework in the  case where $G=\PGL(n,\R)$ and $K=\PO(n)$, and indicate how  statements about statistics of low-lying zeros of principal $L$-functions can be derived from them in  Section \ref{sec:LLZ}. We commence our analysis in Section \ref{sec:2} by establishing  spectral asymptotics for kernels of Hecke operators by means of Fourier integral operators. Based on this, we first describe the asymptotic distribution of Hecke eigenvalues  in Section \ref{sec:4} in the non-equivariant setting, and prove corresponding equidistribution theorems. After this, we turn to the study of the equivariant situation in Section \ref{sec:5}, which besides  the spectral asymptotics for kernels of Hecke operators derived in Section \ref{sec:2} relies on the Fourier inversion formula for orbital integrals. Examples are discussed in the final section. Throughout this paper, we shall use the notation $\N:=\{0,1,2,\dots\}$ and $\N_*:=\{1,2,\dots\}$. Also, we shall write $a \ll_\gamma b$ for two real numbers $a$ and $b$ and a variable $\gamma$, if there exists a constant $C_\gamma>0$ depending only on $\gamma$ such that $|a| \leq C_\gamma b$. If the implicit constant does not depend on any relevant variable, we shall simply write $a\ll b$. \\

{\bf Acknowledgements}. 
The second author would like to thank Tobias Finis, Jasmin Matz, and Werner M\"uller for helpful discussions. He is partially supported by the JSPS Grant-in-Aid for Scientific Research (No. 18K03235).

\section{Asymptotics for Hecke eigenvalues and low-lying zeros for $\PGL(n,\R)$}\label{sec:nonadelic}

In this section, we shall explain our results within a non-adelic framework  in the case where $G=\PGL(n,\R)$,  $K=\PO(n)$,  and  $n\geq 2$, and then apply them to the study of statistics of low-lying zeros of principal $L$-functions following \cite{ST} and \cite{MT}. This section is of an expository nature, and intends to make our results accessible to a wider audience. In fact, although our results are more general, all relevant analytic and geometric aspects are already present in the non-adelic framework. Of course, the expert reader might prefer to directly move to Section \ref{sec:2}.

\subsection{Division algebras and projective linear groups}
Let $D$ be a central division algebra over $\Q$ such that $D\otimes_\Q \R$ is isomorphic to the ring $\M(n,\R)$ of $(n\times n)$-matrices with real entries. 
Fix a specific ring isomorphism $D\otimes_\Q \R\cong \M(n,\R)$, and define  a reduced norm $\mathrm{Nrd}$ on $D$ in terms of the determinant on $\M(n,\R)$. By construction, $\mathrm{Nrd}$ is a homogeneous polynomial with $\Q$-coefficients.
 Let  $H:=\PGL(1,D)$ denote  the projective linear group over $\Q$ defined by $D$, by which one understands the automorphism group $\mathrm{Aut}(D)$ of $D$,  a simple connected algebraic group, see \cite[Chapter VI, Section 23]{Knus-Merkurjev-Rost-Tignol}. Take a maximal order $\mathcal{O}$ in $D$, and choose a $\Z$-basis $e_1=1,e_2,\dots,e_{n^2}$ of $\mathcal{O}$ so that  $\mathcal{O}=\Z e_1+\cdots+\Z e_{n^2}$. Then, a morphism $f:\tilde{H}:=\GL(1,D)\to \SL(n^2-1)$ over $\Q$ is defined by the adjoint action $g\cdot x\mapsto g x g^{-1}$ of $\tilde{H}$ on $ D$, and the elements $e_2,e_3,\dots,e_{n^2}$ of the basis. Here $\GL(1,D)$ means the general linear group over $\Q$ defined by $D$, see \cite[Chapter VI, Section 20]{Knus-Merkurjev-Rost-Tignol}, so that the set $\tilde{H}(F)$ of $F$-rational points in $\tilde{H}$ is the multiplicative group of $D\otimes_\Q F$ for any extension field $F$ of $\Q$. 
The image of $f$ is identified with $H$, and the group of integral points $\Gamma:=H(\Z)=H(\Q)\cap \SL(n^2-1,\Z)$ is a cocompact lattice of the real Lie group $G:=H(\R)\cong \PGL(n,\R)$.
Note that our main theorems are available for arbitrary cocompact arithmetic congruence subgroups, but we consider only  $H(\Z)$ here for simplicity. 

Let $\Q_p$ and $\Z_p$ denote the $p$-adic number field and the ring of integers, respectively, and define the finite adele ring $\bA_\fin$ of $\Q$ as the restricted direct product $\bA_\fin:=\prod_p^\text{rest}\Q_p$, where $p$ moves over all prime numbers.
Note that $\Q$ is regarded as a subring of the adele ring $\bA:=\R\times\bA_\fin$ of $\Q$ via the diagonal embedding.
Set $K_p:=H(\Z_p)$, and consider the open compact subgroup $K_0:=\prod_p K_p$ of $H(\bA_\fin)$, which  satisfies $H(\Q)\cap K_0=\Gamma$.
Since $H(\bA)=H(\Q) (GK_0)$,  it follows that the arithmetic quotient $\Gamma\bsl G$ is topologically isomorphic to $H(\Q)\bsl H(\bA) / K_0$, because $\SL(1,D):=\{g\in \GL(1,D) \mid \mathrm{Nrd}(g)=1\}$ satisfies the Strong Approximation Property to $\infty$ and one has $\mathrm{Nrd}(K_0)=\prod_p \Z_p^\times$, compare \cite{PR}.  It is known that there exists a finite set $S_0$ of primes such that $D\otimes\Q_p$ is isomorphic to $\M(n,\Q_p)$ iff $p$ does not belong to $S_0$.

\subsection{Asymptotics of Hecke eigenvalues}

In what follows, and for the convenience of the reader who is not  familiar with the adelic language, we shall translate our main theorems, which are stated in that framework,  to the non-adelic setting. Let us begin by introducing for each element $\alpha\in H(\Q)\subset \SL(n^2-1,\Q)$ a Hecke operator $T_{\Gamma\alpha\Gamma}$ on the space  $\L^2(\Gamma\bsl G)$ of square integrable functions on $\Gamma \bsl G$ by setting
\[
(T_{\Gamma\alpha\Gamma}\phi)(x):=\sum_{\Gamma\beta\in \Gamma\bsl\Gamma\alpha\Gamma} \phi(\beta x), \qquad \phi\in \L^2(\Gamma\bsl G). 
\]
For details, we refer the reader to Section \ref{sec:specasymphecke}. Consider further the right action of the maximal compact subgroup $K:=\PO(n)\subset G$ on $\Gamma \bsl G$, and recall the Peter-Weyl decomposition 
 $$
 \L^2(\Gamma\bsl G)=\bigoplus_{\sigma\in \widehat{K}}\L^2_\sigma(\Gamma\bsl G)
 $$
  of $\L^2(\Gamma \bsl G)$ into $\sigma$-isotypic components  $\L^2_\sigma(\Gamma\bsl G):=(\L^2(\Gamma\bsl G) \otimes \sigma^\vee)^K$, where  $\sigma^\vee$ denotes the contragredient representation of $\sigma\in \widehat K$, and $\widehat K$ is the unitary dual of $K$.  The operator $T_{\Gamma\alpha\Gamma}$ obviously commutes with the right action of $K$ so that $T_{\Gamma\alpha\Gamma}$ acts on each subspace $\L^2_\sigma(\Gamma\bsl G)$.

From the viewpoint of automorphic representations, there exists an orthonormal basis $\{\phi_j \}_{j\geq 0}$ in $\L^2(\Gamma\bsl G)$ such that each $\phi_j$ is a simultaneous eigenfunction for the Beltrami-Laplace operator $\Delta$ on $G$ and the Hecke operators $T_{\Gamma\alpha\Gamma}$ for every $\alpha\in H(\Q)$ such that the denominators of the entries of $\alpha$ in $\SL(n^2-1,\Q)$ are prime to $\prod_{p\in S_0}p$, where $S_0$ was introduced above.  Their eigenvalues are denoted by
\[
\Delta\phi_j=\lambda_j\phi_j, \qquad T_{\Gamma\alpha\Gamma}\phi_j=\lambda_j(\alpha)\phi_j,
\]
and $0=\lambda_0<\lambda_1\leq \lambda_2\leq\cdots$. 
We may suppose that $\phi_j$ belongs to a single $K$-type. 
Set $\mu_j:=\sqrt{\lambda_j}$, $d:=\dim G=n^2-1$, and $d_\sigma:=\dim\sigma$ for each $K$-type $\sigma$. Our first main result deals with the asymptotic distribution of Hecke eigenvalues and is a special case of Theorems \ref{thm:main} and \ref{thm:equiv}.
\begin{thm}\label{thm:pgl}
For any prime $p\not\in S_0$, $\kappa\in\N$, $\alpha\in H(\Q)\cap \M(n^2-1,p^{-\kappa}\Z)$, and small $\eps>0$,
\[
\sum_{\mu_j\leq \mu } \lambda_j(\alpha) = \delta_\alpha \frac{\vol(\Gamma\bsl G)\, \varpi_d}{(2\pi)^d} \mu^d  +  O(\mu^{d-1}\, p^{n^4\kappa}),
\]
\bqn
\sum_{\stackrel{\mu_j\leq \mu,}{ \phi_j \in \L^2_\sigma(\Gamma\bsl G) }} \lambda_j(\alpha) =  \delta_\alpha \frac{ d_\sigma \,  \vol(\Gamma\bsl G/K) \, \varpi_{d-\dim K} }{(2\pi)^{d-\dim K} }  \mu^{d-\dim K}  +  O_\eps(\mu^{d-\dim K-\frac{d-\dim K-1}{d-\dim K+1}+\eps}\, p^{n^4\kappa}),
\eqn
where $\delta_\alpha:=1$ if $\alpha\in \Gamma$, $\delta_\alpha:=0$ otherwise, $\vol$  denotes the Riemannian volume, and $\varpi_m:=\pi^{\frac m 2}/\Gamma(1+\frac m 2)$ stands for the volume of the unit $m$-sphere.
\qed 
\end{thm}

\subsection{Equidistribution theorems for Satake parameters} As a consequence of Theorem \ref{thm:pgl},  we obtain certain equidistribution statements which are related to the generalized Ramanujan conjecture.   Choose a prime $p\not\in S_0$, and define a subring $\Z(p)$ of $\Q$ by setting $\Z(p):=\{ p^{-l}m \mid m\in\Z$, $l\in\Z\}$.
We have  $G_p:=H(\Q_p)\cong \PGL(n,\Q_p)$ as well as $K_p\cong \PGL(n,\Z_p)$, and in view of  $H(\bA)=H(\Q)(GK_0)$ and the strong approximation property for $\SL(1,D)$ one can prove that the mapping
\[
K_p\bsl G_p/K_p\ni K_p \alpha K_p\, \longmapsto \, H(\Q)\cap K_0\alpha K_0\in \Gamma\bsl H(\Z(p))/\Gamma
\]
is bijective. As a consequence, the unramified Hecke algebra $\H^\mathrm{ur}(K_p\bsl H(\Q_p)/K_p):=C_c^\infty (K_p\bsl G_p/K_p)$ is isomorphic to $C_c(\Gamma\bsl H(\Z(p))/\Gamma)$, and it is known that $\H^\mathrm{ur}(K_p\bsl H(\Q_p)/K_p)$ is generated by the characteristic functions of $K_p \gamma_t K_p$, $1\leq t\leq n-1$, where 
\[
\gamma_t :=\mathrm{diag}(\underbrace{p,\dots,p}_{\text{$t$-times}},1,\dots,1) \, \Q_p^\times \in \PGL(n,\Q_p).
\]
Next, define $\beta_t$ by  $\Gamma\beta_t\Gamma:= H(\Q)\cap K_0\gamma_t^{-1}K_0$, and  the \emph{Satake parameter} of a Hecke-Maass form $\phi_j$ at $p$  as the $n$-tuple  $\alpha^{(j)}(p)=\big (\alpha_1^{(j)}(p),\alpha_2^{(j)}(p),\dots, \alpha_n^{(j)}(p)\big )\in \widehat{T}/\mathfrak{S}_n$ consisting of roots of the equation
\[
x^n-p^{-\frac{n-1}{2}}\lambda_j(\beta_1)x^{n-1}+\cdots +(-1)^t p^{-\frac{(n-t)t}{2}}\lambda_j(\beta_t)x^{n-t}+\cdots+ (-1)^{n-1} p^{-\frac{n-1}{2}}\lambda_j(\beta_{n-1})x + (-1)^n = 0,
\]
where $\widehat{T}:=\{ (u_1,\dots,u_n)\in (\C^\times)^n \mid u_1\cdots u_n=1\}$,  and $\mathfrak{S}_n$ denotes the symmetric group of degree $n$. 
The generalized Ramanujan conjecture predicts that for $j>0$ the Satake parameter $\alpha^{(j)}(p)$  belongs to $\widehat{T}_c/\mathfrak{S}_n$, where $\widehat{T}_c:=\{ (u_1,\dots,u_n)\in \widehat{T} \mid |u_t|=1 \text{ for all }1\leq t\leq n\}$.
We now  introduce the Plancherel measure $\widehat{m}_p^\mathrm{Pl,ur}$ and the Sato-Tate measure $\widehat{m}^\mathrm{ST}$ on $\widehat{T}/\mathfrak{S}_n$. Their supports  are contained in $\widehat{T}_c$, and with respect to the coordinates $(e^{i\theta_1},\dots,e^{i\theta_n})\in\widehat{T}_c$ are defined as  \cite{Sarnak1984}
\[
\widehat{m}_p^\mathrm{Pl,ur} :=c_p \, \prod_{1\leq k<j \leq n}\frac{|e^{i\theta_k}-e^{i\theta_j}|^2}{|p^{-1}e^{i\theta_k}-e^{i\theta_j}|^2} \, \d \theta_1\cdots \d \theta_{n-1},
\]
\[
\widehat{m}^\mathrm{ST} :=c_\infty \, \prod_{1\leq k<j \leq n}|e^{i\theta_k}-e^{i\theta_j}|^2 \, \d \theta_1\cdots \d \theta_{n-1}
\]
 where   $c_p$ and $c_\infty$ are  constants determined by requiring  $\widehat{m}_p^\mathrm{Pl,ur}(\widehat{T}_c/\mathfrak{S}_n)=\widehat{m}^\mathrm{ST}(\widehat{T}_c/\mathfrak{S}_n)=1$. 
 Set 
\[
\mathfrak{F}(\mu):=\{  j\in\N \mid \mu_j \leq \mu  \} ,\qquad \mathfrak{F}_\sigma(\mu):=\{  j\in\N \mid \mu_j \leq \mu , \;\; \phi_j\in \L^2_\sigma(\Gamma\bsl G)  \}. 
\]
The following two corollaries are direct consequences of   Theorem \ref{thm:pgl}, and present some evidence towards the generalized Ramanujan conjecture. They are special cases of Corollaries \ref{plancherel density} and \ref{Sato-Tate equiv}, and Corollaries \ref{cor:equiv.planch}  and \ref{cor:equiv.ST}, respectively.

\begin{cor}[\bf Plancherel density theorem]
Choose a prime $p\not\in S_0$.
Then, one obtains
\[
\lim_{\mu\to\infty}\frac{1}{ |\mathfrak{F}(\mu)|} \sum_{j\in \mathfrak{F}(\mu)} \delta_{\alpha^{(j)}(p)} = \widehat{m}_p^\mathrm{Pl,ur} ,\qquad \lim_{\mu\to\infty}\frac{1}{|\mathfrak{F}_\sigma(\mu)|} \sum_{j\in \mathfrak{F}_\sigma(\mu)}  \delta_{\alpha^{(j)}(p)} = \widehat{m}_p^\mathrm{Pl,ur}.
\]
where $\delta_{\alpha^{(j)}(p)}$ denotes the Dirac delta measure at $\alpha^{(j)}(p)\in \widehat{T}/\mathfrak{S}_n$.
\qed
\end{cor}
\begin{cor}[\bf Sato-Tate equidistribution theorem]
Choose a prime $p\not\in S_0$, and let $\{(p_k,\mu_k)\}_{k\geq 1}$ be a sequence such that $p_k\to \infty$ and $p_k^l / \mu_k \to 0$ as $k\to \infty$ for any integer $l\geq 1$.
Then
\[
\lim_{k\to\infty}\frac{1}{ |\mathfrak{F}(\mu_k)|} \sum_{j\in \mathfrak{F}(\mu_k)} \delta_{\alpha^{(j)}(p_k)} = \widehat{m}^\mathrm{ST} ,\qquad \lim_{k\to\infty}\frac{1}{|\mathfrak{F}_\sigma(\mu_k)|} \sum_{j\in \mathfrak{F}_\sigma(\mu_k)}  \delta_{\alpha^{(j)}(p_k)} = \widehat{m}^\mathrm{ST}.
\]
\qed
\end{cor}

\subsection{Low-lying zeros of principal $L$-functions} \label{sec:LLZ}

As another consequence of Theorem \ref{thm:pgl} we are able to give some results about the statistics of low-lying zeros of $L$-functions.  By the strong multiplicity-one theorem, each $\phi_j$ determines a unique automorphic representation $\pi_j=\otimes_v \pi_{j,v}$ of $H$, to which one can associate  the principal $L$-function 
\bqn 
L(s,\pi_j):=\prod_p L_p(s,\pi_{j,p})
\eqn
 by Godement-Jacquet theory, see \cite{GJ,Jacquet,Badulescu} for details. In our context,  for each prime $p\not\in S_0$, the local $L$-factor $L_p(s,\pi_{j,p})$ is given in terms of  the Satake parameters by 
\[
L_p(s,\pi_{j,p}):=(1-\alpha_1^{(j)}(p) \, p^{-s})^{-1}(1-\alpha_2^{(j)}(p) \, p^{-s})^{-1}\cdots (1-\alpha_n^{(j)}(p) \, p^{-s})^{-1}.
\]
Furthermore, $L(s,\pi_j)$ can be analytically continued to an entire function on $\C$ if $j>0$. If $j=0$, $\phi_0$ is  constant on $\Gamma \bsl G$, and consequently on $\Gamma\bsl G/K$, every $\pi_{0,v}$ equals the trivial representation, and  $L(s,\pi_0)=\prod_{t=1}^n\zeta(s+\frac{n+1}{2}-t)$, where $\zeta(s)$ is the Riemann zeta function.  As a consequence, and following \cite{Sfamily} and \cite{SST}, we can regard $\mathfrak{F}_\sigma(\mu)$ as a family of $L$-functions.
Notice that an automorphic representation may contain  Hecke-Maass forms $\phi_j$ belonging to several $K$-types $\sigma$. Later,  we will introduce  families of automorphic representations in a more general setting following \cite{ST},  $\mathfrak{F}_\sigma(\mu)$ being an instance of such a family in the present context.
 
In order  to study the statistics of low-lying zeros of the $L$-functions $L(s,\pi_j)$ belonging to  the family $\mathfrak{F}_\sigma(\mu)$ one introduces the \emph{average analytic conductor} 
\[
\log C(\mathfrak{F}_\sigma(\mu)):=\frac{1}{|\mathfrak{F}_\sigma(\mu)|} \sum_{j\in  \mathfrak{F}_\sigma(\mu)} \log C(\pi_j) ,
\]
where $C(\pi_j)$ denotes the analytic conductor of $\pi_j$, and we refer to \cite{IS2000}
 and \cite[Chapter 5]{IK} for its definition.
It is obvious that $C(\mathfrak{F}_\sigma(\mu)) \asymp \mu^n$ as $\mu\to\infty$. Choose a Paley-Wiener function $\Phi$ on $\C$ whose Fourier transform $\widehat{\Phi}$ has sufficiently small support on $\R$, and define the \emph{average one-level density} of the family $\mathfrak{F}_\sigma(\mu)$ as
\[
\mathcal{D}_1(\mathfrak{F}_\sigma(\mu);\Phi) := \frac{1}{|\mathfrak{F}_\sigma(\mu)|} \sum_{j\in \mathfrak{F}_\sigma(\mu)} \sum_{ \rho_j=\frac{1}{2}+i \gamma_j }  \Phi\left(  \frac{\gamma_j}{2\pi} \log C(\mathfrak{F}_\sigma(\mu)) \right) ,
\]
where $\rho_j$ ranges over all non-trivial zeros of $L(s,\pi_j)$. We then have the following 
\begin{thm}\label{thm:lowlyingzeros} If $n\geq 3$, 
\[
\lim_{\mu\to \infty} \mathcal{D}_1(\mathfrak{F}_\sigma(\mu);\Phi) =\int_{-\infty}^{\infty} \Phi(x) \, W(\mathrm{U})(x) \, \d x = \widehat{\Phi}(0),
\]
while if $n=2$,
\[
\lim_{\mu\to \infty} \mathcal{D}_1(\mathfrak{F}_\sigma(\mu);\Phi) =\begin{cases} \int_{-\infty}^{\infty} \Phi(x) \, W(\mathrm{SO}_\mathrm{even})(x) \, \d x & \text{if $\sigma$ is trivial,} \\ \int_{-\infty}^{\infty} \Phi(x) \, W(\mathrm{SO}_\mathrm{odd})(x) \, \d x & \text{if $\sigma=\det$,}  \\ \int_{-\infty}^{\infty} \Phi(x) \, W(\mathrm{O})(x) \, \d x & \text{if $\sigma$ is $2$-dimensional.} \end{cases} 
\]
Here the density functions $W$ are given by
\[
W(\mathrm{U})=1, \quad W(\mathrm{SO}_\mathrm{even})=1+\frac{\sin 2\pi x}{2\pi x}, \quad W(\mathrm{SO}_\mathrm{odd})=1-\frac{\sin 2\pi x}{2\pi x}+\delta_0(x),
\]
\[
W(\mathrm{O})=\frac{1}{2}W(\mathrm{SO}_\mathrm{even})+\frac{1}{2}W(\mathrm{SO}_\mathrm{odd}).
\]
\end{thm}
\begin{proof}
If $\sigma$ is $1$-dimensional,  the statement is essentially contained in \cite[Theorem 2.1]{MT} for general $n$, and for $n=2$ in \cite[Theorem 1.4]{AM}, because our family can be obtained from the families considered there by restricting finitely many local components of automorphic representations. As for the case $n=2$, the statement can be reduced to the case $\dim \sigma=1$ by considering $K$-types of  irreducible unitary representations of $G$. 
Hence, we are left with the  task of proving the assertion for $n>2$ and $\dim \sigma>1$.
Now,  \cite{Badulescu} implies that  for $j>0$ the automorphic representation of $\PGL(n,\Q)$ associated to a Hecke-Maass form $\phi_j$  is cuspidal, so that  $\mathfrak{F}_\sigma(\mu)$ is essentially cuspidal, compare also  \cite{SST}.
In addition, Theorem \ref{thm:pgl} implies that $\mathfrak{F}_\sigma(\mu)$ has rank zero in the sense of \cite{SST}. Consequently, the assertion follows with the same arguments that were used  in \cite[Section 12]{ST}, see also \cite[Section 2]{MT}.
\end{proof}

For the non-equivariant family $\mathfrak{F}(\mu)$, we can also define the average analytic conductor $C(\mathfrak{F}(\mu))$ and the average $1$-level density $\mathcal{D}_1(\mathfrak{F}(\mu);\Phi)$ as above.
The multiplicity of a $K$-type $\sigma$ of a parabolically induced representation of $G$ is bounded by the dimension of $\sigma$, cf. \cite[p. 293]{Mueller}, because for $G=\PGL(n,\R)$, any Levi subgroup of a cuspidal parabolic subgroup is isomorphic to a product of copies of $\GL(1,\R)$ and $\GL(2,\R)$ modulo the center.
Therefore, for a single $\pi_j$, the number of Hecke-Maass forms $\phi\in \pi_j$ with  $\mu_\phi\leq \mu$ increases by the order $\mu^3$ for $n>2$ or $\mu$ for $n=2$, where $\mu_\phi>0$ is defined by $\Delta\phi=\mu_\phi^2 \phi$.  
Consequently, they can be neglected in the total growth of $\mathfrak{F}(\mu)$, and we obtain $C(\mathfrak{F}(\mu)) \asymp \mu^n$.
Thus, by the same argument as in the proof of Theorem \ref{thm:lowlyingzeros} for $n\geq 3$, we obtain
\bqn
\lim_{\mu\to \infty} \mathcal{D}_1(\mathfrak{F}(\mu);\Phi) = \begin{cases}\int_{-\infty}^{\infty} \Phi(x) \, W(\mathrm{U})(x) \, \d x & \text{if } n\geq 3, \\ \int_{-\infty}^{\infty} \Phi(x) \, W(\mathrm{O})(x) \, \d x &\text {if } n=2.\end{cases} 
\eqn

\section{Spectral asymptotics for kernels of Hecke operators}\label{sec:2}

In this section, we begin our analysis by deriving spectral asymptotics for kernels of Hecke operators by means of Fourier integral operators. 

\subsection{Non-equivariant spectral asymptotics}\label{sec:2.1}
Let $M$ be  a closed  Riemannian manifold $M$ of dimension $d$  and $P_0$ an elliptic classical pseudodifferential operator on $M$ of degree $m$, which is assumed to be  positive and symmetric. Denote its  unique self-adjoint extension by  $P$, and let $\mklm{\phi_j}_{j\geq 0}$ be an orthonormal basis of $\L^2(M)$ consisting of eigenfunctions of $P$ with eigenvalues $\mklm{\lambda_j}_{j \geq 0}$ repeated according to their multiplicity. 
Let $p(x,\xi)$ be the principal symbol of $P_0$, which is strictly positive and  homogeneous in $\xi$ of degree $m$ as a function  on $T^\ast M\setminus\mklm{0}$, that is,  the cotangent bundle of $M$ without the zero section.  Here and in what  follows  $(x,\xi)$ denotes an element in $T^*Y \simeq Y \times \R^d$ with respect to the canonical trivialization of the cotangent  bundle over a chart domain $Y\subset M$. Consider further the $m$-th root $Q:=\sqrt[m]{P}$ of $P$ given by the spectral theorem. It is well  known that $Q$ is a classical pseudodifferential operator of order $1$ with principal symbol $q(x,\xi):=\sqrt[m]{p(x,\xi)}$ and the first Sobolev space as domain. Again, $Q$ has discrete spectrum, and its eigenvalues  are given by $\mu_j:=\sqrt[m]{\lambda_j}$.  The spectral properties of $P$ can  be described by  studying the \emph{spectral function} of $Q$, which in terms of the basis $\mklm{\phi_j}$ is given by
\bq
\label{eq:specfunct} 
e(x,y,\mu):=\sum_{\mu_j\leq \mu} \phi_j(x) \overline{\phi_j(y)},
\eq
and belongs to $\Cinft(M \times M)$ as a function of $x$ and $y$ for any $\mu \in \R$. Let $s_\mu$ be the spectral projection onto the sum of eigenspaces of $Q$ with eigenvalues in the interval  $(\mu, \mu+1]$, and denote its Schwartz kernel by 
$$
s_\mu(x,y):=e(x,y,\mu+1) - e(x,y,\mu).
$$
 To obtain an asymptotic description of the spectral function of $Q$, let $\rho \in \S(\R,\R_+)$ be such that $\hat \rho(0)=1$ and $\supp \hat \rho\in (-\delta/2,\delta/2)$ for an arbitrarily small  $\delta>0$, and define the {approximate spectral projection operator} 
\bqn 
\widetilde s_\mu u := \sum_{j=0}^\infty \rho(\mu-\mu_j) E_{j} u, \qquad u \in \L^2(M),
\eqn
where $E_j$ denotes the orthogonal projection onto the subspace spanned by $\phi_j$. Clearly, 
\bq
\label{eq:29.5.2017}
K_{\widetilde s_\mu}(x,y):=\sum_{j=0}^\infty \rho(\mu-\mu_j) \phi_j(x) \overline{\phi_j(y)}\in \Cinft(M\times M)
\eq
 constitutes the Schwartz kernel of $\widetilde s_\mu$. Describing $\tilde s_\mu$ as a  Fourier integral operator 
  one obtains the following 
 
\begin{proposition}
\label{prop:31.05.2017}
{\rm \cite[Proposition 3.1]{RW}}
Suppose  that the cospheres $S^\ast_xM:=\mklm{(x,\xi) \in T^\ast M\mid  p(x,\xi)=1}$ are strictly convex.\footnote{This condition holds, for example, if $P_0=\Delta$ equals the Beltrami--Laplace operator, since then $p(x,\xi)=\|\xi\|^2_x$.}
Then,  as $\mu \to +\infty$, for any  fixed $x, y \in M$,  and $\tilde N=0,1,2,3,\dots$ one has the expansion\footnote{For the case $x=y$ see also \cite[Proposition 2.1]{duistermaat-guillemin75}.}
\begin{align*}
\begin{split}
K_{\widetilde s_\mu}(x,y)&=   \mu^{d-1-{\frac{\delta_{x,y}}{2}}} \left [ \sum_{r=0}^{\tilde N-1} \Lcal_r(x,y)\, \mu^{-r} + \cR_{\tilde N}(x,y,\mu) \right ]
\end{split}
\end{align*}
up to terms of order $O(\mu^{-\infty})$, where  
\bqn 
\delta_{ x,y}:=\begin{cases} 0, & y=x, \\ d-1, & y \not=x.  \end{cases}
\eqn
The coefficients in the expansion and  the remainder $\cR_{\tilde N}(x,y,\mu)=O_{x,y}(\mu^{-\tilde N})$ term can be computed explicitly;  if $y=x$, they  are uniformly bounded in $x$ and $y$,  while  if $y \not=x$, they satisfy the bounds
 \begin{align*}
 \label{eq:31.5.2017}
\begin{split}
\Lcal_r(x,y) &\ll \, {\dist (x,y)}^{-(d-1)/2-r},  \qquad \cR_{\tilde N}(x,y,\mu) \ll \, \dist (x, y)^{-(d-1)/2-\tilde N} \, \mu^{-\tilde N},
\end{split}
\end{align*}
where $\dist(x,y)$ denotes the geodesic distance between two points  belonging to the same connected component, while  $\dist(x,y):=\infty$ for points in different components. On the other hand, $K_{\widetilde s_\mu}(x,y)$ is rapidly decreasing as $\mu\to -\infty$. 
\end{proposition}  
\qed

 To describe the leading term more explicitly, note that by \cite[(3.5)]{RW}
\bqn 
K_{\widetilde s_\mu}(x,x)= \frac{\mu^d}{(2\pi)^{d+1}} \int_\R\int_\R e^{i\mu(t-Rt)} I(\mu,R,t,x) \d R \d t
\eqn 
up to terms of order $O(\mu^{-\infty})$, where $I(\mu,R,t,x)$ is a  compactly supported smooth function in $(R,t)$ satisfying $I(\mu,1,0,x)=\vol S^\ast_x(M)$. If we now apply the stationary phase theorem to the above oscillatory integral with phase function $t-Rt$ we obtain
\bq
\label{eq:28.3.2018}
K_{\widetilde s_\mu}(x,x)= \frac{\mu^{d-1}}{(2\pi)^{d}} \vol S^\ast_x(M) + O(\mu^{d-2})
\eq
as $\mu \to + \infty$, the only critical point being $(R,t)=(1,0)$. This could also be read off directly from \cite[(2.2)]{duistermaat-guillemin75}.
 
 \subsection{Equivariant spectral asymptotics}  Keeping the notation as above, assume now  that $M$ carries an  isometric action of a compact  Lie group $K$, and consider  the right regular representation $\pi$ of $K$ on $\L^2(M)$ with corresponding 
 Peter-Weyl decomposition
\bq
\label{eq:PW}
\L^2(M) = \bigoplus_{\sigma \in \widehat K} \L^2_\sigma(M), \qquad \L^2_\sigma(M):= \Pi_\sigma \L^2(M),
\eq
where $\widehat K$ denotes the unitary dual of $K$, which we identify with the set of characters of $K$, and 
\bqn 
\Pi_\sigma:= d_\sigma \int_K \overline{\sigma(k)} \pi(k) \d k
\eqn
 the orthogonal projector onto the $\sigma$-isotypic component $\L^2_\sigma(M)$, $dk$ being Haar measure and $d_\sigma$ the dimension of an irreducible representation $(V_\sigma,\pi_\sigma)$ of $K$ in the class $\sigma \in \widehat K$. Note that  $\L^2_\sigma(M)\simeq(L^2(M)\otimes \sigma^\vee)^K$, where $(L^2(M)\otimes \sigma^\vee)^K=(L^2(M)\otimes V_\sigma)^K$ consists of   $L^2$-functions $\phi:M\rightarrow V_\sigma$  that are $K$-equivariant in the sense that $\phi(m \cdot k)=\pi_\sigma(k)^{-1} \phi(m)$. The components of $\phi$  as $L^2$-functions from $M$ to $\C$ correspond then to elements  in $\L^2(M)_\sigma$.  Further, suppose that $P$ commutes with $\pi$, and that the orthonormal basis $\mklm{\phi_j}_{j \geq 0}$ is compatible with the decomposition \eqref{eq:PW} in the sense that each $\phi_j$ lies in some $\L^2_\sigma(M)$. Then every eigenspace of $P$ is invariant under $\pi$, and decomposes into irreducible $K$-modules spanned by  eigenfunctions. The fine structure of the spectrum of $P$ is described by the spectral function of the operator $Q_\sigma:=\Pi_\sigma \circ Q \circ \Pi_\sigma= \Pi_\sigma \circ Q= Q \circ \Pi_\sigma$, which is also called the \emph{reduced spectral function}, and given by 
 \bq
 \label{eq:02.02.2017} 
 e_\sigma(x,y,\mu):=\sum_{\mu_j \leq \mu, \, \phi_j \in \L^2_\sigma(M)} \phi_j(x) \overline{ \phi_j(y)}. 
 \eq
To study it, one considers the composition $s_\mu \circ \Pi_\sigma$, or rather $
\tilde s_\mu \circ \Pi_\sigma$, whose kernel has the spectral expansion
\bq
\label{eq:24.11.2016b}
K_{\widetilde s_\mu \circ \Pi_\sigma}(x,y)=\sum_{j \geq 0,  \phi_j \in \L^2_\sigma(M)} \rho(\mu-\mu_j) \phi_j(x) \overline{\phi_j(y)}.
\eq
Write $\O_x:=x \cdot K$ for the $K$-orbit through $x$. 
Similarly to  Proposition \ref{prop:31.05.2017}, using Fourier integral operator methods one  proves  the following\footnote{Note that the additional assumption made in \cite[Section 3.2]{RW} and  \cite{ramacher16} that $K$ acts effectively on $M$ is unnecessary. } 
  
\begin{proposition}{\rm \cite[Proposition 3.3]{RW}}
\label{prop:30.01.2017}
Suppose that $K$ acts on $M$ with orbits of the same dimension $\kappa\leq d-1$ and  that the cospheres $S^\ast_xM:=\mklm{(x,\xi) \in T^\ast M\mid  p(x,\xi)=1}$ are strictly convex. Then, for any  fixed $x, y \in M$, $\sigma \in \widehat K$,  and $\tilde N=0,1,2,3,\dots$ one has the expansion
\begin{align*}
\begin{split}
K_{\widetilde s_\mu \circ \Pi_\sigma}(x,y)&=    \mu^{d-{\frac{\epsilon_{x,y}}{2}}-1} d_\sigma  \left [ \sum_{r=0}^{\tilde N-1} \Lcal_r^\sigma (x,y)\, \mu^{-r} + \cR^\sigma_{\tilde N}(x,y,\mu) \right ]
\end{split}
\end{align*}
up to terms of order $O(\mu^{-\infty})$ as $\mu \to +\infty$, where  
\bqn 
\varepsilon_{ x,y}:=\begin{cases} 2 \kappa, & y \in \O_x, \\ d-1+\kappa, & y \notin \O_x.  \end{cases}
\eqn
The coefficients in the expansion and  the remainder term can be computed explicitly;  if $y \in \O_x$, they  satisfy the bounds
 \begin{align*}
\begin{split}
\Lcal^\sigma_r(x,y ) &\ll \, \sup_{u \leq 2r}\norm{D^u \sigma}_\infty, \qquad
 \cR^\sigma_{\tilde N}(x,y,\mu) \ll \, \sup_{u \leq 2\tilde N+ \lfloor \frac \kappa 2+1\rfloor}\norm{D^u \sigma}_\infty \, \mu^{-\tilde N},
\end{split}
\end{align*}
uniformly in $x$ and $y$, where $D^u$ denote differential operators on $K$ of order $u$, and  if $y \notin \O_x$, the bounds
 \begin{align*}
\begin{split}
\Lcal^\sigma_r(x,y ) &\ll \, \sup_{u \leq 2r}\norm{D^u \sigma}_\infty \cdot {\dist (x, \O_y)}^{-\frac{d-\kappa-1}2-r}, \\ 
 \cR^\sigma_{\tilde N}(x,y,\mu) &\ll \, \sup_{u \leq 2\tilde N + \lfloor\frac{\kappa}2+1\rfloor}\norm{D^u \sigma}_\infty \cdot \dist (x, \O_y)^{-\frac{d-\kappa-1}2-\tilde N} \, \mu^{-\tilde N},
\end{split}
\end{align*}
where $\dist (x, \O_y):=\min\mklm{\dist(x,z)\mid z \in \O_y}$. On the other hand, $K_{\widetilde s_\mu\circ \Pi_\sigma}(x,y)$ is rapidly decreasing as $\mu\to -\infty$.
\end{proposition} 
\qed
 
As far as the leading term is concerned, by \cite[Proposition 4.1]{ramacher16} one has as $\mu \to +\infty$
\bq
\label{eq:12.4.2018}
K_{\widetilde s_\mu\circ \Pi_\sigma}(x,x)= \frac{\mu^{d-\kappa-1}}{(2\pi)^{d-\kappa}} d_\sigma [\pi_{\sigma|K_x}:\1] \vol [(\Omega \cap S^\ast_x(M))/ K] + O(\mu^{d-\kappa- 2}),
\eq
where $[\pi_{\sigma|K_x}:\1]$ is a Frobenius factor that denotes the multiplicity of the trivial representation in the restriction of $\pi_\sigma$ to the stabilizer $K_x$ of $x$,  and $\Omega$ is the zero level of the momentum map corresponding to  the Hamiltonian $K$-action on $T^\ast M$. 
 
\subsection{Spectral asymptotics for Hecke operators}
\label{sec:specasymphecke} In what follows, we shall apply the previous considerations to derive asymptotics for kernels of Hecke operators in the eigenvalue aspect. To introduce the setting, let $G$ be a $d$-dimensional real semisimple Lie group  with finite center  and Lie algebra $\g$.  Denote by  $\langle X,Y\rangle := \tr \, (\ad X\circ \ad Y)$ the Cartan-Killing form   on $\g$ and by $\theta$   a Cartan involution  of $\g$. Let 
\bq
\label{eq:cartan}
\g = \k\oplus\p
\eq
be  the Cartan decomposition of $\g$ into the eigenspaces of  $\theta$, corresponding to the eigenvalues  $+1$ and $-1$ , respectively, and denote the  maximal compact subgroup of $G$ with Lie algebra $\k$ by $K$. Put $\langle X,Y\rangle _\theta:=-\langle X,\theta Y\rangle $. Then $\langle \cdot,\cdot \rangle _\theta$ defines a left-invariant Riemannian metric on $G$ with corresponding distance function $\dist_G$. 

Now, let   $\Gamma_1$, $\Gamma_2,\dots,\Gamma_h$ be discrete cocompact subgroups of $G$ which are mutually commensurable.   The set $\Gamma: =\cap_{l=1}^h \Gamma_l$ is a subgroup of  finite index and the disjoint union\footnote{In this paper, the symbol $\coprod$ will denote the disjoint union of possibly intersecting sets. If $G$ is not compact and one identifies the sets $\Gamma_j \bsl G$ with fundamental domains in $G$, the latter can be chosen such that they have no intersections, since they are bounded.}
\[
M:=\Gamma_1\bsl G \coprod \Gamma_2\bsl G \coprod \cdots \coprod \Gamma_h\bsl G \simeq \{   (g,l) \mid 1\leq l\leq h, \;\; g\in \Gamma_l \bsl G     \}
\]
is a closed  manifold, where each  point in $x \in M$ can be expressed as a pair $(g,l)\equiv \Gamma_lg$ of a representative $g\in G$ and the subscript of $\Gamma_l$. The left-invariant metric on $G$ induces a Riemanniann metric  and a  distance function   $\dist$ on each of the connected components $\Gamma_l\bsl G$ of $M$   according to 
\bqn 
\dist(\Gamma_l g, \Gamma_l h):=\inf_{\gamma \in \Gamma_l} \dist_G(g,\gamma h),
\eqn
 while for $l\not=j$ one sets $\dist((g,l),(h,j)):=\infty$. In order to introduce  Hecke operators on $M$ we consider the commensurator of $\Gamma$
\[
C(\Gamma):=\{  g\in G \mid \text{$\Gamma$ is commensurable with $g^{-1}\Gamma g$} \}.
\]
Since $\Gamma\bsl \Gamma \alpha \Gamma\simeq (\Gamma \cap \alpha \Gamma \alpha^{-1}) \bsl  \alpha \Gamma \alpha^{-1}$, for each element $\alpha\in C(\Gamma)$  one has\footnote{In this paper, the symbol $\sqcup$ will denote union of disjoint sets. For disjoint sets, the operations $\coprod$ and $\bigsqcup$ coincide.} 
\bq
\label{eq:06.04.2018}
\Gamma_j\alpha\Gamma_l=\bigsqcup_{u=1}^{U_{\alpha, j,l}}\Gamma_j\beta_{u}, \qquad U_{\alpha, j,l} \in \N_\ast, \, \beta_u \in \Gamma_j \alpha \Gamma_l, 
\eq
so that it is natural to define a linear mapping $T_{\Gamma_j\alpha\Gamma_l}:\L^2(\Gamma_j\bsl G)\to \L^2(\Gamma_l\bsl G)$ by the expression 
\[
(T_{\Gamma_j\alpha\Gamma_l}f)(g,l):=\sum_{u=1}^{U_{\alpha, j,l}}f(\beta_u g,j)=:\sum_{\beta\in \Gamma_j\bsl \Gamma_j\alpha\Gamma_l}f(\beta g,j).
\]
\begin{rem}[Notation]
\label{rem:4.5.19}
According to general convention,  $\beta \equiv \Gamma_j \beta$ (resp. $\beta_u\equiv \Gamma_j \beta_u $) denotes both a right coset as well as a suitable representative in $\Gamma_j\alpha\Gamma_l\subset G$, and the products $\beta g$ and $\beta_u g$ are taken in $G$, compare \cite[Section 2.8]{Miyake}.
\end{rem}
 Note that the so-called \emph{Hecke points} $(\beta_u g,j)$ do depend on the representative $g$, while the sums defining $T_{\Gamma_j\alpha\Gamma_l}$ do not depend on the representatives $g$ and $\beta_u$. In fact, for a different representative $g_1$, the Hecke points $(\beta_ug_1,j)$ are given by a permutation of the points $(\beta_ug ,j)$. We now generalize this definition, and introduce for each tuple  $\alpha\equiv (\alpha_{j,l,m})_{1\leq j,l\leq h, \; 1\leq m \leq c_{j,l}}$ with  $\alpha_{j,l,m}\in C(\Gamma)$ and  $c_{j,l}\in\N$ a \emph{Hecke operator} $T_\alpha:=(\sum_{m=1}^{c_{j,l}}T_{\Gamma_j\alpha_{j,l,m}\Gamma_l})_{1\leq j,l\leq h}$ on 
\bq \label{eq:l2}
\L^2(M):=\left\{  \varphi:G\times \{1,2,\dots,h\}\to\C \mid \begin{array}{l} \text{$\varphi$ is measurable,} \\ \text{$\varphi(\gamma g,j)=\varphi(g,j)$ holds for any $\gamma\in \Gamma_j$, $g\in G$,} \\ \text{$\sum_{j=1}^h \int_{\Gamma_j\bsl G} |\varphi(g,j)|^2dg<\infty$}  \end{array} \right\}
\eq
 by setting\footnote{Note that $c_{j,l}=0$ corresponds to the trivial mapping from $\L^2(\Gamma_j\bsl G)$ to $\L^2(\Gamma_l\bsl G)$. Also, as subsets in $G$ the $\Gamma_j\alpha_{j,l,m}\Gamma_l$ are not disjoint in general.} 
\[
(T_\alpha f)(g,l):=\sum_{j=1}^h \sum_{m=1}^{c_{j,l}} (T_{\Gamma_j\alpha_{j,l,m}\Gamma_l}f)(g,l)=\sum_{j=1}^h \sum_{m=1}^{c_{j,l}} \sum_{\beta\in \Gamma_j\bsl \Gamma_j\alpha_{j,l,m}\Gamma_l} f(\beta g,j). 
\]

Next, let $P_0$ be an elliptic left-invariant differential operator on $G$ of degree $m$ which gives rise to a positive and symmetric operator $P$ on $\L^2(M)$ with strictly convex cospheres $S_x^*(M)$. With $\widetilde s_\mu$ as in Section \ref{sec:2.1} we obtain for the Schwartz kernel of $T_\alpha \circ \widetilde s_\mu$ the expression 
\bq \label{eq:4}
K_{T_\alpha\circ \widetilde s_\mu} (x,x)=\sum_{j=1}^h \sum_{m=1}^{c_{j,l}} \sum_{\beta\in \Gamma_j\bsl \Gamma_j\alpha_{j,l,m}\Gamma_l}K_{\widetilde s_\mu} ((\beta g , j),(g,l)), \qquad x=(g,l)\in M. 
\eq

As a consequence of Proposition \ref{prop:31.05.2017} we now deduce 
\begin{lem}\label{lem:non-equiv}
Choose a Hecke operator $T_\alpha$ on $\L^2(M)$ given by a tuple $\alpha\equiv(\alpha_{j,l,m})_{1\leq j,l\leq h, \; 1\leq m \leq c_{j,l}}$ as above. Set
\[
\delta_{\alpha,l}:=\sharp\{ m \mid  1\leq m \leq c_{l,l} , \, \Gamma_l\subset \Gamma_l\alpha_{l,l,m}\Gamma_l   \}.
\]
Then, for each $x=(g,l)\in M$ one has as $\mu \to +\infty$

\begin{align*}
&  K_{T_\alpha\circ \widetilde s_\mu} (x,x) - \delta_{\alpha,l} \, K_{ \widetilde s_\mu}(x,x)  \\
& =  O\Big (\mu^{(d-1)/2} \, D(\alpha,x) \,  \sum_{m=1}^{c_{l,l}} |\Gamma_l\bsl \Gamma_l\alpha_{l,l,m}\Gamma_l|+\mu^{-\infty}\,  \sum_{j=1}^h \sum_{m=1}^{c_{j,l}} |\Gamma_j\bsl \Gamma_j\alpha_{j,l,m}\Gamma_l|\Big )
\end{align*}
where
\[
D(\alpha,x):=  \max_{1\ne\beta\in \bigcup_{m=1}^{c_{l,l}} \Gamma_l\bsl \Gamma_l\alpha_{l,l,m}\Gamma_l}\dist(\Gamma_l \beta g,\Gamma_l g)^{-(d-1)/2}.
\] 
\end{lem}
\begin{proof}
To begin, note that by definition of the distance in $\Gamma_l\bsl G$ there exists a constant $c_{x,\alpha}>0$ such that 
\[
\min_{1\ne\beta\in\bigcup_{m=1}^{c_{l,l}} \Gamma_l \bsl \Gamma_l\alpha_{l,l,m}\Gamma_l}\dist(\Gamma_l\beta g,\Gamma_l g) > c_{x,\alpha}.
\]
Consequently, we infer for  any $x=(g,l)\in M$ from Proposition \ref{prop:31.05.2017} and \eqref{eq:4} that
\[
 K_{T_\alpha\circ \widetilde s_\mu} (x,x) - \delta_{\alpha,l} \,  K_{ \widetilde s_\mu}(x,x)  \ll  \mu^{(d-1)/2}   \sum_{m=1}^{c_{l,l}} \sum_{1\not=\beta\in \Gamma_l\bsl \Gamma_l\alpha_{l,l,m}\Gamma_l}\dist((\beta g ,l),(g,l))^{-(d-1)/2} 
\]
up to terms of order $O(\mu^{-\infty})$ times the cardinality of the sum in \eqref{eq:4}, since $\dist((\beta g,j),(g,l))=\infty$ if $j\neq l$. 
\end{proof}

Next, let  $K$ be any compact subgroup of $G$, and recall that $K$ acts on $G$ and  each $\Gamma_j \bsl G$ from the right  in an isometric  and effective way,  the isotropy group of a point $\Gamma_j g\in \Gamma_j \bsl G$ being conjugate 
to the finite group $gKg^{-1}\cap \Gamma_j$. Hence, all $K$-orbits in $\Gamma_j\bsl G$ are either principal or exceptional, and of dimension $\dim K$. Since the maximal compact subgroups of $G$ are precisely the conjugates of $K$, exceptional $K$-orbits arise from elements in $\Gamma_j $ of finite order. Consider the right regular representation $\pi$ of $K$ on $\L^2(M)$ together with the corresponding  Peter-Weyl decomposition \eqref{eq:PW}, and suppose that  $P_0$ commutes with $\pi$ and  the Hecke operators $T_\alpha$, which commute with the right regular $K$-representation as well. The  Schwartz kernel of $T_\alpha \circ \widetilde s_\mu\circ \Pi_\sigma $ is then given by the expression 
\bq \label{eq:4cis}
K_{T_\alpha\circ \widetilde s_\mu\circ \Pi_\sigma} (x,x)=\sum_{j=1}^h \sum_{m=1}^{c_{j,l}} \sum_{\beta\in \Gamma_j\bsl \Gamma_j\alpha_{j,l,m}\Gamma_l}K_{\widetilde s_\mu\circ \Pi_\sigma} ((\beta g , j),(g,l)), \qquad x=(g,l)\in M. 
\eq

As a consequence of Proposition \ref{prop:30.01.2017} one now deduces the following generalization of Lemma \ref{lem:non-equiv}.

\begin{lem}\label{lem:15.04.2018}
Choose a Hecke operator $T_\alpha$ on $\L^2(M)$ given by a tuple $\alpha\equiv(\alpha_{j,l,m})_{1\leq j,l\leq h, \; 1\leq m \leq c_{j,l}}$, and consider for each fixed point $x=(g,l) \in M$ the sets of Hecke points
\begin{align*}
 H(\alpha,x)&:= \coprod_{m=1}^{c_{l,l}} \mklm{y=(\beta g,l) \in \Gamma_l\bsl G \mid \, \beta\in  \Gamma_l\bsl \Gamma_l\alpha_{l,l,m}\Gamma_l},\\
 T(\alpha,x)&:=  \mklm{y=(\beta g,l) \in  H(\alpha,x)  \mid  \beta =g k_y g^{-1} \in g K g^{-1}},  \\
 C(\alpha,x)&:=  \mklm{y=(\beta g,l) \in H(\alpha,x) \mid \, \beta\in \Gamma(\alpha,l)} \subset T(\alpha,x),
\end{align*}
where the element $k_y\in K$ is uniquely determined by the Hecke point $y$, and
we put
\begin{align*}
\Gamma(\alpha,l)&:= \bigcup_{m=1}^{c_{l,l}} \{\beta\in \Gamma_l\bsl \Gamma_l\alpha_{l,l,m}\Gamma_l \mid \beta \in K \cap  C(G) \},
\end{align*}
$C(G)$ being the center of $G$. One then has for each  $\sigma \in \widehat K$ the asymptotic formula
\begin{align*}
&  K_{T_\alpha\circ \widetilde s_\mu\circ \Pi_\sigma} (x,x) -  \left [ \sum_{k \in  \Gamma(\alpha,l)}\sigma(k) +\sum_{y \in T(\alpha,x)-C(\alpha,x)}\sigma(k_y)\right ] K_{\widetilde s_\mu\circ \Pi_\sigma} (x,x) \\
& =  O\bigg ( \mu^{(d-\dim K -1)/2}  D_K(\alpha,x)  \sum_{m=1}^{c_{l,l}} |\Gamma_l\bsl \Gamma_l\alpha_{l,l,m}\Gamma_l|  +\mu^{-\infty}\,   \sum_{j=1}^h \sum_{m=1}^{c_{j,l}} |\Gamma_j\bsl \Gamma_j\alpha_{j,l,m}\Gamma_l|\bigg )
\end{align*}
as $\mu \to + \infty$, where
\[
D_K(\alpha,x):=  \max_{y \in H(\alpha,x)-T(\alpha,x)}\dist(y K, xK)^{-(d-\dim K-1)/2}.
\]
\end{lem}
\begin{rem}
\label{rem:2.6}
In the statement of the lemma, keep in mind that according to Remark \ref{rem:4.5.19} the symbol $\beta\equiv \Gamma_l \beta$ denotes both the coset $\Gamma_l \beta$ as well as a suitable representative $\beta$. In this sense, the relations $\beta=g k_y g^{-1}$ and $\beta \in K \cap C(G)$ are to be understood that they are valid for a suitable representative. Furthermore,  taking $K=\mklm{1}$ one recovers Lemma \ref{lem:non-equiv}. 
\end{rem}

\begin{proof}
Since $K$-orbits are closed, one has for $g,g_1 \in G$ the equivalences
\begin{align*}
\dist(\Gamma_l g K, \Gamma_l g_1 K)=0 &\quad \Longleftrightarrow \quad \Gamma_l g K \cap  \Gamma_l g_1 K\neq \emptyset  \\ & \quad \Longleftrightarrow \quad \Gamma_l g K = \Gamma_l g_1 K   \\ 
& \quad \Longleftrightarrow \quad \text{there exist $\gamma \in \Gamma_l$ and $k \in K$ such that $\gamma g k =g_1$}.
\end{align*}
Consequently, one deduces  for any  $ \beta\in \sqcup_{m=1}^{c_{l,l}} \Gamma_l\bsl \Gamma_l\alpha_{l,l,m}\Gamma_l$, eventually after choosing a suitable representative, the implications
\bqn 
\dist(\Gamma_l g K,\Gamma_l \beta g K) =0  \quad \Longleftrightarrow \quad \beta= g k_\beta g^{-1} \text{ for some  $k_\beta \in K$}.
\eqn
From Proposition \ref{prop:30.01.2017} and  \eqref{eq:4cis}  we then infer  for any $x=(g,l)\in M$ that
\begin{multline}\label{eq:20190224}
 K_{T_\alpha\circ \widetilde s_\mu\circ \Pi_\sigma} (x,x) -  \sum_{y \in   T(\alpha,x)}K_{\widetilde s_\mu\circ \Pi_\sigma} (y,x)   \\
 \ll  \mu^{(d-\dim K -1)/2}  \sum_{y \in H(\alpha,x)- T(\alpha,x)}\dist(yK ,xK)^{-(d-\dim K - 1)/2} 
\end{multline}
up to terms of order $O(\mu^{-\infty})$ times the cardinality of the sum in \eqref{eq:4cis}, since $\dist((\beta g,j),(g,l))=\infty$ if $j\neq l$. Now, as a consequence of the $K$-equivariance of the kernel of $\widetilde s_\mu\circ \Pi_\sigma$ one deduces for any 
$y =(\beta g,l) \in T(\alpha,x)$ the equality 
\begin{align*}
K_{\widetilde s_\mu\circ \Pi_\sigma} (y,x)&=
\sum_{j \geq 0,  \phi_j \in \L^2_\sigma(M)} \rho(\mu-\mu_j) \phi_j(\Gamma_l gk_y ) \overline{\phi_j(\Gamma_l g)}=\sigma(k_y) K_{\widetilde s_\mu\circ \Pi_\sigma} (x,x),
\end{align*}
so that 
\bqn 
\sum_{y\in  T(\alpha,x)}K_{\widetilde s_\mu\circ \Pi_\sigma} (y,x)=  \sum_{y\in  T(\alpha,x)}\sigma(k_y)  K_{\widetilde s_\mu\circ \Pi_\sigma} (x,x).
\eqn 
Since for $y=(\beta g,l) \in C(\alpha,x)$ one has $k_y=\beta$, the assertion of the lemma  follows. 
\end{proof}

In the remaining of this section, let us assume that $K$ is the maximal compact subgroup given by the Cartan decomposition \eqref{eq:cartan}, in which case  $C(G) \subset K$.

\begin{lem}\label{lem:22.5.2019}
For $\beta \in G$ set $N(\beta, K):=\{ h\in G \mid h^{-1}\beta h\in K\}$,  $C_\beta':=\mklm{h^{-1} \beta h \mid h \in G}$, and assume that  $C_\beta' \cap K\not=\emptyset$. Then $N(\beta,K)$ is an analytic manifold. Moreover, if $G$ has no compact simple factors,
\begin{align*}
 \dim N(\beta,K)=d \quad \Longrightarrow \quad \beta \in C(G) \quad \Longrightarrow \quad N(\beta,K)=G.
\end{align*}
\end{lem}
\begin{proof}
By assumption we have  $\beta=gk_0 g^{-1}$ for some $g \in G$, $k_0 \in K$. Then $N(\beta,K)= g N(k_0,K)$, and 
\bqn
C_{k_0}' \cap K = \mklm{h^{-1} k_0 h \in K \mid h \in G}\simeq G_{k_0} \bsl N(k_0,K),
\eqn
where $G_{k_0}=\mklm{h \in G \mid hk_0h^{-1}=k_0}$. By \cite[Theorem 3.1]{richardson67}, $C_{k_0}' \cap K$ is an  analytic manifold, and consequently also $N(k_0,K)$ and $N(\beta,K)$, proving the first assertion. Next, $\beta \in C(G)$ implies that $N(\beta,K)=G$ has dimension $d$. Conversely, assume that $N(\beta,K)$ has dimension $d$, and consider the global Cartan decomposition corresponding to  \eqref{eq:cartan}, which is given by the diffeomorphism
\bq
\label{eq:globalCartan}
 \p \times K  \ni  (X,k) \quad \longmapsto \quad \exp{X} \cdot  k=g \in G. 
\eq
 Then 
 \begin{align*}
 N(k_0,K)&\simeq  \mklm{(X,k) \in \p \times K  \mid \exp (-X) \cdot k_0 \cdot  \exp X \in K} 
 \end{align*}
 has dimension $d$. Next, note that for arbitrary $Y,Z \in \g$, and $h \in G$ one has $h \cdot \exp (Y) \cdot h^{-1}=\exp (\Ad(h) Y)$, as well as
 \bqn 
 \Ad(\exp(-sZ)) Y = Y-s[Z,Y]+O(s^2),
 \eqn
 provided that $s \in \R$ has small absolute value, see \cite[pp. 127 and 128]{helgason78}.  Now, let $h=\exp X \cdot k \in N(k_0,K)$ be arbitrary and $U_h \subset N(k_0,K)$ an open neighborhood of $h$. By assumption, $U_h$ is $d$-dimensional, so that 
 \bq
 \label{eq:22.5.2019}
 \exp X \cdot \exp X_1 \cdot k \in U_h \qquad \text{for all $X_1 \in \p$ with $\norm{X_1}$ sufficiently small.} 
 \eq
 By assumption, $\exp(-X) \cdot k_0 \cdot \exp X=k_1$ for some $k_1 \in K$. If $K$ is connected, the exponential map from $\k_0$ to $K$ is onto, so that we can write $k_1=\exp Y_1$ for some $Y_1 \in \k$.  In view of $[\p,\p] \subset \k$, $[\k,\p] \subset \p$, and $C(G) \subset K$ we conclude that for almost all $X_1 \in \p$ with $\norm{X_1}$ sufficiently small
  \begin{align*}
 \exp(-X_1) \cdot \exp(-X) \cdot k_0 \cdot  \exp X \cdot \exp X_1 &=  \exp(-X_1) \cdot  \exp Y_1 \cdot \exp X_1 \\
 & = \exp (Y_1 -[X_1,Y_1]+O(\norm{X_1}^2)) \notin K,
 \end{align*}
 unless $k_0 \in C(G)$. 
Note that here we have used the assumption that $G$ has no compact simple factors. If $K$ is not connected,   we may  suppose that $G\subset \SL(N,\R)$ and $K\subset \SO(N)$ for some sufficiently large $N \in \N$,  and repeat the above arguments using the surjectivity of the exponential onto $\SO(N)$ and the Cartan decomposition of $\SL(N,\R)$. 
 Since by  \eqref{eq:22.5.2019} we must have $\exp X \cdot \exp X_1 \cdot k \in N(k_0,K)$, we conclude that $\beta=k_0 \in C(G)$, completing the proof. \end{proof}

Using the previous lemma one deduces 

\begin{lem}\label{lem:zero}
In the situation of Lemma \ref{lem:15.04.2018}, suppose that $G$ has no compact simple factors. Then the function defined by 
\bqn 
F(x):=\sum_{y \in T(\alpha,x)-C(\alpha,x)} \sigma(k_y)
\eqn
if $ T(\alpha,x)-C(\alpha,x)\neq \emptyset$, $F(x):=0$ else,   is supported on a set of measure zero in $M$.
\end{lem}
\begin{proof}
To begin, let $U\subset G$ be a sufficiently small  open neighbourhood of the identity and $h \in U$. For $x=(g,l)$ and $x_1=(gh,l) \in M$, one has the one-to-one correspondence of Hecke points
 \bqn 
(\beta g,l) \in H(\alpha,x)  \qquad \stackrel{1:1}\longleftrightarrow \qquad (\beta g h,l) \in H(\alpha,x_1);
 \eqn
furthermore, points in $C(\alpha,x)$ correspond to points in $C(\alpha,x_1)$. 
Now,  take any  $y=(\beta g,l) \in H(\alpha,x)-C(\alpha,x)$. Then, by Lemma \ref{lem:22.5.2019},  $\dim N(\gamma \beta,K) < d$ for all  $\gamma\in \Gamma_l$. 
  In order that   $y_1=(\beta gh, l) \in T(\alpha,x_1)-C(\alpha,x_1)$ we must have $\gamma \beta=gh k_{y_1} h^{-1} g^{-1}$ for some $\gamma \in \Gamma_l$ and  $k_{y_1} \in K$, which is equivalent to 
\bqn 
h \in N(g^{-1} \gamma \beta g,K)=g^{-1} N(\gamma \beta,K). 
\eqn
That is, $h$ must belong to a lower dimensional set in $U$. In other words, Hecke points in $T(\alpha,x_1)-C(\alpha,x_1)$ can only arise from Hecke points in $H(\alpha,x)-C(\alpha,x)$ by deformation along a measure zero set. 
 Consequently, if $x \in \supp F$ we can only have $x_1 \in \supp F$ if $h$ belongs to a measure zero set in $U$, and the assertion follows.
\end{proof}
As a consequence of the previous lemma, 
 \bqn 
\int_{\Gamma_l\bsl G} \left [ \sum_{k\in \Gamma(\alpha,l)}\sigma(k) +F(x) \right ] K_{\widetilde s_\mu\circ \Pi_\sigma} (x,x) \d x= \sum_{k\in \Gamma(\alpha,l)}\sigma(k) \int_{\Gamma_l\bsl G}  K_{\widetilde s_\mu\circ \Pi_\sigma} (x,x) \d x,
\eqn
so that non-central torsion elements do not contribute to the leading term in Lemma \ref{lem:15.04.2018} after integration over $\Gamma_l\bsl G$.

\section{Non-equivariant asymptotics  for Hecke eigenvalues and Sato-Tate equidistribution}\label{sec:4}

We commence our study of the asymptotic distribution of Hecke eigenvalues by considering first  the non-equivariant setting. 

\subsection{Preliminaries} Let $H$ denote a semisimple connected linear algebraic group over the rational number field $\Q$.
We may suppose that $H$ is a closed subgroup of $\SL(N)$ over $\Q$ for a fixed $N\in\N_*$.
We write $\Q_p$ for the $p$-adic number field, and $\bA$ (resp. $\bA_\fin$) for the adele (resp. finite adele) ring of $\Q$.
We choose an open compact subgroup $K_0$ of $H(\bA_\fin)$.
As usual, we regard $H(\Q)$ as a subgroup of $H(\bA)$ by the diagonal embedding.
By the finiteness of class numbers \cite[Theorems 5.1 and 8.1]{PR}, there exist elements $x_1=1,x_2,\dots,x_{c_H}$ in $H(\bA_\fin)$ such that
\bq \label{eq:d}
H(\bA)=\bigsqcup_{l=1}^{c_H} H(\Q)x_lH(\R)K_0.
\eq
There exists also a finite set $S_0$ of primes such that
\begin{itemize}
\item[(C1)] $H$ is unramified over $\Q_p$ for every prime $p\not\in S_0$;
\item[(C2)] one has $K_0=K_{S_0} \prod_{p\not\in S_0}K_p$, where  $K_{S_0}$ is an open compact subgroup of $\prod_{p\in S_0} G(\Q_p)$ and $K_p=H(\Q_p)\cap \SL(N,\Z_p)$ is a hyperspecial compact subgroup of $H(\Q_p)$ for every $p\not\in S_0$;
\item[(C3)] for $x_l=(x_{l,p})_p$, we have $x_{l,p}\in K_p$ for any $p\not\in S_0$. In other words, we may suppose that $x_{l,p}=1$ for all $ p\not\in S_0$ without loss of generality.
\end{itemize}
For details, we refer to \cite{Tits}, and  normalize the Haar measures on $H(\bA_\fin)$ and $H(\Q_p)$ by setting $\vol(K_0)=1$ and $\vol(K_p)=1$ for all $p\not\in S_0$.
 
 Next,  fix a prime $p\not\in S_0$. The group $G_p:=H(\Q_p)$ has a Borel subgroup that contains a maximal $\Q_p$-torus $T_p$, and its  Cartan decomposition reads $G_p=K_p T_p K_p$.
Let $A_p$ denote the maximal $\Q_p$-split subtorus in $T_p$.
Since the inclusion mapping $A_p\subset T_p$ induces an isomorphism $A_p/A_p\cap K_p \cong T_p/T_p\cap K_p$, one gets $G_p=K_p A_p K_p$.
Let $X_*(A_p)$ denote the abelian group of co-characters of $A_p$.
A hight function $\|\cdot  \|_p$ on $G_p$ is defined by $\|g\|_p:=\max_{i,j}\{|g_{i,j}|_p,|g_{i,j}'|_p\}$ for $g=(g_{i,j})_{1\leq i,j\leq N}\in G_p\subset \SL(N,\Q_p)$ and $g^{-1}=(g_{i,j}')_{1\leq i,j\leq N}$, where $|\cdot |_p$ denotes the valuation of $\Q_p$.
Note that $\| k_1gk_2 \|_p=\| g \|_p$ holds for any $k_1,k_2\in K_p$ and $g\in G_p$.
Further, we define a hight function $\|\cdot \|_p$ on $X_*(A_p)$ by
\[
\| \omega \|_p := \left| \frac{\log\| \omega(p) \|_p}{\log p}\right| \in \N, \qquad \omega\in X_*(A_p),
\]
as well as the unramified Hecke algebra $\mathcal{H}^\mathrm{ur}(G_p):=C^\infty_c(K_p\bsl G_p/K_p)$, which  is generated by the family of characteristic functions $\tau_\omega$ of the double cosets $K_p\omega(p)K_p$ with ${\omega\in X_*(A_p)}$.  Also, for each $\kappa\in\N$, a truncated unramified Hecke algebra $\mathcal{H}^\mathrm{ur}_\kappa(G_p)$ is defined by
\[
\mathcal{H}^\mathrm{ur}_\kappa(G_p):=\langle \tau_\omega \mid  \quad \omega\in X_*(A_p) ,\,   \|\omega\|_p\leq\kappa \rangle.
\]
For other, essentially equivalent definitions of  $\mathcal{H}^\mathrm{ur}_\kappa(G_p)$ see \cite[Section 2.3]{ST}, \cite[Section 3.4]{Marshall2017}, and  \cite[Lemma 3.5]{Marshall2017}.

In what follows, we write $G_p^{\wedge,\mathrm{ur}}$ (resp. $G_p^{\wedge,\mathrm{ur,temp}}$) for the unramified (resp. unramified and tempered) part of the unitary dual of $G_p$.
Let $\Omega_p$ denote the $\Q_p$-rational Weyl group for $(G_p,A_p)$.
By the canonical map given in \cite[pp. 33--34]{ST}, we have the topological injective mapping $G_p^{\wedge,\mathrm{ur}} \to \widehat{A_p}/\Omega_p$ and the topological isomorphism
\[
G_p^{\wedge,\mathrm{ur,temp}} \cong \widehat{A_p}_c/\Omega_p
\]
where $\widehat{A_p}$ denotes the dual torus of $A_p$ and $\widehat{A_p}_c$ denotes the compact subtorus of  $\widehat{A_p}$.
For $f\in \mathcal{H}^\mathrm{ur}(G_p)$, a continuous function $\widehat{f}$ on $\widehat{A_p}/\Omega_p$ is defined by
\[
\widehat{f}(\pi):=\Tr \pi(f), \qquad \pi \in G_p^{\wedge,\mathrm{ur}},
\]
and it is well-known that the Plancherel measure $\widehat{m}_p^\mathrm{Pl,ur}$ on $G_p^{\wedge,\mathrm{ur}}$ satisfies
\bq \label{eq:localplan}
\widehat{m}_p^\mathrm{Pl,ur}(\widehat{f})=f(1) , \qquad  f\in \mathcal{H}^\mathrm{ur}(G_p). 
\eq
Notice that the support of $\widehat{m}_p^\mathrm{Pl,ur}$ is included in $G_p^{\wedge,\mathrm{ur,temp}} \cong \widehat{A_p}_c/\Omega_p$.
For explicit descriptions of the Plancherel measure $\widehat{m}_p^\mathrm{Pl,ur}$, we refer to \cite{MacdonaldBull} and \cite[Proposition 3.3]{ST}.

Next, let $S$ be a finite set of prime numbers.
Suppose that $S$ has no intersection with $S_0$, and  set
\[
\Q_{S}:=\prod_{p\in S}\Q_p \quad \text{and} \quad \mathcal{H}^\mathrm{ur}_\kappa(H(\Q_{S})):=\bigotimes_{p\in S}\mathcal{H}^\mathrm{ur}_\kappa(H(\Q_p)).
\]
Each element $(g_p)_{p\in S}\in H(\Q_S)$ is identified with the element $(y_v)_{v<\infty}\in H(\bA_\fin)$ such that $y_p=g_p$ for all $ p\in S$ and $y_v=1$  for all $v\not\in S$.
Write $H(\Q_S)^{\wedge,\mathrm{ur}}$ (resp. $H(\Q_S)^{\wedge,\mathrm{ur,temp}}$) for the unramified (resp. unramified and tempered) part of the unitary dual of $H(\Q_S)$. 
Clearly, there is an injective mapping $H(\Q_S)^{\wedge,\mathrm{ur}} \to \prod_{p\in S} \widehat{A_p}/\Omega_p$, and one has an isomorphism
\bqn
 H(\Q_S)^{\wedge,\mathrm{ur,temp}}\cong \prod_{p\in S}\widehat{A_p}_c/\Omega_p.
\eqn
Further, for each $f_S\in \mathcal{H}^\mathrm{ur}(H(\Q_{S}))$  define the continuous function 
\bqn
\widehat{f_S}: \,  H(\Q_S)^{\wedge,\mathrm{ur}} \ni \pi_S \longmapsto   \widehat{f_S}(\pi_S):=\Tr\pi_S(f_S).
\eqn
Since the Plancherel measure $\widehat{m}_S^\mathrm{Pl,ur}$ on $H(\Q_S)^{\wedge,\mathrm{ur}}$ satisfies $\widehat{m}_S^\mathrm{Pl,ur}=\prod_{p\in S} \widehat{m}_p^\mathrm{Pl,ur}$, one has  $\widehat{m}_S^\mathrm{Pl,ur}(\widehat{f_S})=f_S(1)$ by \eqref{eq:localplan} and the support of $\widehat{m}_S^\mathrm{Pl,ur}$ is contained in $H(\Q_S)^{\wedge,\mathrm{ur,temp}}$. 

We shall now recall briefly some fundamental facts about Sato-Tate measures for $H$, and refer the reader to \cite[Section 5]{ST} for details. 
Denote by $\widehat{H}$ the dual group of $H$  and by $\widehat{T}$ the maximal torus in $\widehat{H}$, which is a constituent of the root datum, compare  \cite{BorelLfct}. The Galois group $\Gal(\overline{\Q}/\Q)$ acts on $\widehat{H}$ via its natural action on the root datum, and there exists a finite extension $F_1$ of $\Q$ such that $\Gal(\overline{\Q}/\Q)$ acts on $\widehat{H}$ through the faithful action of $\Gamma_1:=\Gal(F_1/\Q)$.
Let $\widehat{K_1}$ be a $\Gamma_1$-invariant maximal compact subgroup of $\widehat{H}$, set $\widehat{T}_c:=\widehat{T}\cap \widehat{K_1}$, and let $\Omega_c$ denote the Weyl group for $(\widehat{K_1},\widehat{T}_c)$.
For each $\theta\in\Gamma_1$, set
\[
\widehat{T}_{c,\theta}:=\widehat{T}_c/(\theta-\mathrm{id})\widehat{T}_c , \qquad \Omega_{c,\theta}:=\{ w\in \widehat{T}_c \mid \theta(w)=w\},
\]
and denote by $\widehat{m}^\mathrm{ST}_{\theta}$  the $\theta$-Sato-Tate measure  on $\widehat{T}_{c,\theta}/ \Omega_{c,\theta}$ introduced in \cite[Definition 5.1]{ST}. It can be characterized by a limit of Plancherel measures as follows. Write $\mathscr{C}(\Gamma_1)$ for a set of representatives of conjugacy  classes in $\Gamma_1$, and consider the corresponding  partition of the set of primes outside $S_0$ 
\[
\{ \mathcal{V}(\theta)\}_{\theta\in\mathscr{C}(\Gamma_1)}.
\]
 Fix $\theta\in\mathscr{C}(\Gamma_1)$, and for each $p\in \mathcal{V}(\theta)$ choose an inclusion $\overline{\Q}\hookrightarrow \overline{\Q_p}$ such that the Frobenius $\mathrm{Fr}_p$ in $\Gal(\Q_p^\mathrm{ur}/\Q_p)$ has image $\theta$ in $\Gamma_1$, yielding  the identification
\bq \label{eq:ST(5.2)}
\widehat{T}_{c,\theta}/ \Omega_{c,\theta} = \widehat{T}_{c,\mathrm{Fr}_p}/ \Omega_{c,\mathrm{Fr}_p}\cong G_p^{\wedge,\mathrm{ur,temp}},
\eq
see \cite[(5.2)]{ST}. By \cite[Proposition 5.3]{ST} one then has the weak convergence
\[
\widehat{m}_p^\mathrm{Pl,ur}\to \widehat{m}^\mathrm{ST}_{\theta} \quad \text{as $p\to \infty$ in $\mathcal{V}(\theta)$.}
\]
Consequently, there is a unique Sato-Tate measure $\widehat{m}^\mathrm{ST}$, which coincides with the limit $\lim_{p\to \infty}\widehat{m}_p^\mathrm{Pl,ur}$ in the weak topology.

\subsection{Non-equivariant asymptotics and equidistribution results}\label{Non-equivariant case}
In this paper we will mainly deal with the case where $H(\Q)\bsl H(\bA)$ is compact, which we assume from now on. In this situation,  $\L^2(H(\Q)\bsl H(\bA))$ decomposes into a countable orthogonal direct sum of irreducible unitary representations $\pi$ of $H(\bA)$, so that 
\[
\L^2(H(\Q)\bsl H(\bA))\cong \bigoplus_{\pi\in\widehat{H(\bA)}}  V_\pi^{\oplus m_\pi},
\]
where $\widehat{H(\bA)}$ denotes the unitary dual of $H(\bA)$,  $m_\pi\in\N$ the multiplicity of $\pi, $ and $V_\pi$  a representative space of $\pi$, see \cite{GGP}.
For each double coset $K_0\alpha K_0$ with $\alpha\in H(\bA_\fin)$, a Hecke operator $T_{K_0\alpha K_0}$ on $\L^2(H(\Q)\bsl H(\bA)/K_0)$ is defined by
\[
(T_{K_0\alpha K_0}\phi)(x):=\sum_{\beta\in K_0\alpha K_0/K_0}\phi(x\beta), \qquad \phi\in \L^2(H(\Q)\bsl H(\bA)/K_0).
\]

Write $G:=H(\R)$, and let $K$ be a maximal compact subgroup of $G$, and $\Delta$ the Beltrami-Laplace operator on $G$.
It is known that $G$ is a $d$-dimensional semisimple real Lie group with finite center \cite{PR} and that $\Delta=-\mathcal{C}+2\mathcal{C}_K$, where $\mathcal{C}$ (resp. $\mathcal{C}_K$) denotes the Casimir operator of $G$ (resp. $K$), compare \cite{RW}.
We choose an orthonormal basis $\{\phi_j\}_{j\in\N}$ in $\L^2(H(\Q)\bsl H(\bA)/K_0)$ such that each $\phi_j$ is a $\Delta$-eigenfunction included in a single space $V_\pi$.
Since any automorphic representation $\pi$ factors as a tensor product $\pi=\otimes_v\pi_v$  of irreducible unitary representations $\pi_v$ of $H(\Q_v)$ for all places $v$ of $\Q$ \cite{Flath}, any $\phi_j$ is a simultaneous eigenfunction for $\Delta$ and $T_{K_0\alpha K_0}$ for any $\alpha\in H(\bA_\fin^{S_0})$, where
\[
\bA_\fin^{S_0}:=\{(\alpha_v)_{v<\infty}\in \bA_\fin \mid  \alpha_p= 1 \quad  \forall \, p\in S_0 \}.
\]
Let $\lambda_j$ and $\lambda_j(\alpha)$ denote the eigenvalue of $\phi_j$ for $\Delta$  and  $T_{K_0\alpha K_0}$ respectively, so that  
\[
\Delta\phi_j=\lambda_j\phi_j \quad \text{and} \quad T_{K_0\alpha K_0}\phi_j=\lambda_j(\alpha)\phi_j, \qquad \alpha\in H(\bA_\fin^{S_0}).
\]
Set $\mu_j:=\sqrt{\lambda_j}$.
Our goal is to study the asymptotics of the sum
\bq \label{eq:b}
\sum_{\mu_j\leq \mu } \lambda_j(\alpha) ,\qquad \alpha\in H(\bA_\fin^{S_0}),
\eq
of Hecke eigenvalues with respect to the spectral parameter $\mu\in \R_{>0}$. For this, let $1\leq l\leq c_H$ and set
\bq \label{eq:20190622}
\Gamma_l :=H(\Q)\cap x_lK_0x_l^{-1},
\eq
where $c_H$ and $x_l$ are as in \eqref{eq:d}, and we regard $H(\Q)$ as a subgroup of $H(\bA_\fin)$ via the diagonal embedding.
Then, one gets a diffeomorphism
\[
M := \coprod_{l=1}^{c_H} \Gamma_l\bsl G \cong \bigsqcup_{l=1}^{c_H} \Gamma_l\bsl G \cdot x_l \cong H(\Q)\bsl H(\bA)/K_0,
\]
which defines an isomorphism from $\L^2(M)$ to $\L^2(H(\Q)\bsl H(\bA)/K_0)$ given by
the mapping 
\begin{gather*}
\L^2(M)\ni \varphi_M\mapsto \varphi\in \L^2(H(\Q)\bsl H(\bA)/K_0), \\
\varphi(\gamma x_l g_\infty k_0):=\varphi_M(g_\infty,l) , \quad \gamma\in H(\Q),\;\; g_\infty\in G, \;\; k_0\in K_0.
\end{gather*}
\begin{lem}\label{lem:2}
For any element $\alpha$ in $H(\bA_\fin)$, there exist elements $\alpha_{j,k,m}\in H(\Q)$  with $1\leq m\leq c_{j,k}$ such that
\bq \label{eq:c}
H(\Q)\cap x_j K_0\alpha^{-1} K_0 x_k^{-1} = \bigsqcup_{m=1}^{c_{j,k}} \Gamma_j \alpha_{j,k,m} \Gamma_k.
\eq
\end{lem}
\begin{proof}
By \eqref{eq:d}, there obviously exist $\gamma_{jm}\in H(\Q)$, $R\in \N_*$, $1\leq n_m \leq c_H$ such that $x_j K_0\alpha^{-1} K_0 =\bigsqcup_{m=1}^{R} \gamma_{jm} x_{n_m}K_0$. For this reason, it is sufficient to show that 
\[
H(\Q)\cap x_j K_0\alpha^{-1} K_0 x_k^{-1} = \bigcup_{1\leq m\leq R,\; x_{n_m}=x_k} \Gamma_j \gamma_{jm} \Gamma_k.
\]
If  $x_{n_m}=x_k$, one gets $\gamma_{jm}\in x_j K_0 \alpha^{-1} K_0 x_k^{-1}$,  hence  $\Gamma_j \gamma_{jm} \Gamma_k \subset x_j K_0\alpha^{-1} K_0 x_k^{-1}$. It follows that $H(\Q)\cap x_l K_0\alpha^{-1} K_0 x_m^{-1} \supset \bigcup_{1\leq j\leq R,\; x_{n_j}=x_m} \Gamma_l \gamma_{lj} \Gamma_m$.
Next, we suppose that  $\gamma\in H(\Q)\cap x_j K_0\alpha^{-1} K_0 x_k^{-1}$.
For some $1\leq j\leq R$ and $k_0\in K_0$, one has $\gamma x_k= \gamma_{jm} x_{n_m} k_0$.
This implies  $x_{n_m}=x_k$ by \eqref{eq:d}, and therefore  $\gamma=\gamma_{jm} x_k k_0 x_k^{-1}\in \gamma_{jm}\Gamma_m$.
Thus, we obtain $H(\Q)\cap x_j K_0\alpha^{-1} K_0 x_k^{-1} \subset \bigcup_{1\leq m\leq R,\; x_{n_m}=x_k} \Gamma_j \gamma_{jm} \Gamma_k$, and the assertion follows.
\end{proof}

\begin{lem}\label{lem:relation}
Let $H(\bA_\fin) \ni \alpha\equiv(\alpha_{j,k,m})_{1\leq j,k\leq c_H,\, 1\leq m\leq c_{j,k}}$ be as in  Lemma \ref{lem:2}. In terms of the isomorphism  $\L^2(M)\simeq  \L^2(H(\Q)\bsl H(\bA)/K_0)$,  $T_{K_0\alpha K_0}$ coincides with  the Hecke operator $T_\alpha$ defined in Section \ref{sec:specasymphecke}, and 
\bq\label{eq:27.1.2020}
T_{K_0\alpha K_0}\varphi=T_\alpha \varphi_M , \quad  \text{where} \;\;  (T_\alpha \varphi_M)(g,k):= \sum_{j=1}^{c_H}\sum_{m=1}^{c_{j,k}} ( T_{\Gamma_j\alpha_{j,k,m}\Gamma_k}\varphi_M)(g,j).
\eq
\end{lem}
\begin{proof}
For $x_k$ and $\beta\in K_0\alpha K_0/K_0$, there exists an element $x_j$ such that $x_j \beta K_0\cap H(\Q)x_k K_0\neq \emptyset$ by \eqref{eq:d}.
Hence, one has
\[
\varphi(g_\infty x_j \beta)=\varphi((\gamma)_\fin g_\infty x_k ) = \varphi((\gamma^{-1})_\infty g_\infty x_k ) 
\]
for some $\gamma\in H(\Q)$, where $(\gamma)_\fin$ (resp. $(\gamma)_\infty$) denotes the embedding of $\gamma$ into $H(\bA_\fin)$ (resp. $H(\R)$). 
For this reason, we have only to prove the one-to-one correspondence between the left cosets of $K_0\bsl K_0\alpha^{-1} K_0$ and the left cosets of  $\sqcup_{j=1}^{c_H}\sqcup_{m=1}^{c_{j,k}}\Gamma_j\bsl \Gamma_j \alpha_{j,k,m} \Gamma_k$, but this is obvious because the left $x_jK_0x_j^{-1}$-equivalence coincides with the left $\Gamma_j$-equivalence in $H(\Q)\cap x_j K_0\alpha^{-1} K_0 x_k^{-1}$.
\end{proof}

We can now state the first main result of this paper. 
Set
\[
p_{S}:=\prod_{p\in S}p \quad  \text{and} \quad K_S:=\prod_{p\in S}K_p .  
\]
For each $\alpha\in H(\Q_S)$, write $\|\alpha\|_S\leq \kappa$ if the characteristic function of $K_S\alpha K_S$ belongs to $\mathcal{H}^\mathrm{ur}_\kappa(H(\Q_{S}))$.
\begin{thm}[\bf Asymptotic distribution of Hecke eigenvalues]\label{thm:main}
Let $\{ \phi_j \}_{j\in\N}$ be an orthonormal basis of $\L^2(H(\Q)\bsl H(\bA)/K_0)$  as above and $d:=\dim H(\R)$.
Then there exists a constant $0<c<d+N(N+1)$ such that for any finite set $S$ of primes in the complement of  $S_0$ and any $\alpha\in H(\Q_S)$ with $\|\alpha\|_S\leq \kappa$ one obtains
\[
\sum_{\mu_j\leq \mu } \lambda_j(\alpha) = \delta_\alpha \frac{\vol(M) \, \varpi_d }{(2\pi)^d} \mu^d  +  O(\mu^{d-1}\, p_S^{c\kappa}),
\]
where $\delta_\alpha:=1$ if $\alpha\in K_S$, $\delta_\alpha:=0$ otherwise, $\vol(M)$ denotes the Riemannian volume of $M$, and $\varpi_d:=\pi^{\frac d 2}/\Gamma(1+\frac d 2)$ means the volume of the unit $d$-sphere.
\end{thm}

\begin{proof}
The assertion is essentially a consequence of Lemma \ref{lem:non-equiv}. As a consequence of the two previous lemmata, any  $\alpha\in H(\bA_\fin)$ can be identified with a tuple $(\alpha_{j,k,m})_{1\leq j,k\leq c_H,\, 1\leq m\leq c_{j,k}}$ up to $(\Gamma_j,\Gamma_k)$-equivalence via the decomposition \eqref{eq:c}, and be associated to a Hecke operator $T_\alpha$ on $\L^2(M)$ as in  \eqref{eq:27.1.2020}. By Lemma \ref{lem:2} we can assume for each pair $(j,k)$,  that  $\Gamma_j\alpha_{j,k,m_1}\Gamma_k\neq\Gamma_j\alpha_{j,k,m_2}\Gamma_k$ if $m_1\neq m_2$, where  $1\leq m_1,m_2\leq c_{j,k}$. Also, assume that either of the conditions 
\begin{itemize}
\item[(i)] $\Gamma_l$ lies in $\sqcup_{m=1}^{c_{l,l}}\Gamma_l\alpha_{l,l,m}\Gamma_l$ for every $l$,
\item[(ii)] $\Gamma_l$ does not lie in  $\sqcup_{m=1}^{c_{l,l}}\Gamma_l\alpha_{l,l,m}\Gamma_l$ for any $l$,
\end{itemize}
holds, and set $\delta_\alpha':=1$ if (i)   and $\delta_\alpha':=0$ if (ii) is fulfilled. Then, since the doble cosets $\Gamma_l\alpha_{l,l,m}\Gamma_l$ are disjoint in the present case, Lemma \ref{lem:non-equiv} implies  for each $x=(g,l)\in M$ that
\begin{multline}\label{eq:non-equiv}
  K_{T_\alpha\circ \widetilde s_\mu} (x,x) - \delta_\alpha' \, K_{ \widetilde s_\mu}(x,x)  \\
=  O\Big (\mu^{(d-1)/2} \, D(\alpha,x) \,  \sum_{m=1}^{c_{l,l}} |\Gamma_l\bsl \Gamma_l\alpha_{l,l,m}\Gamma_l|+\mu^{-\infty}\,  \sum_{j=1}^{c_H} \sum_{m=1}^{c_{j,l}} |\Gamma_j\bsl \Gamma_j\alpha_{j,l,m}\Gamma_l|\Big ).
\end{multline}
 Denote by $\mathrm{lcm}(\gamma)$ the least common multiple of denominators of components of  a matrix $\gamma\in  H(\Q)\subset \SL(N,\Q)$, and consider  an element $\alpha\in H(\Q_S)$ with $\|\alpha\|_S\leq \kappa$. By  Lemma \ref{lem:2}, there is a constant $c_1\in\N_*$ such that $\mathrm{lcm}(\gamma) < c_1 p_S^{\kappa}$ holds for any $\gamma\in \sqcup_{m=1}^{c_{j,k}}\Gamma_j \alpha_{j,k,m}\Gamma_k$, $1\leq j,k\leq c_H$.
Hence, for some constant $c_2$ and $c_3=N(N+1)$ one has
\bq
\label{eq:9.6.19}
\sum_{j=1}^{c_H} \sum_{m=1}^{c_{j,k}} \sharp(\Gamma_j\bsl \Gamma_j \alpha_{j,k,m} \Gamma_k)< c_2 \,  p_S^{c_3\kappa},
\eq
because $\sqcup_{m=1}^{c_{j,k}}\Gamma_j \alpha_{j,k,m}\Gamma_k$ is contained in $\SL(N,\Q)\cap \M(N,c_1^{-1}p_S^{-\kappa}\Z)$. 
Furthermore, for any $x=(g,k)\in M=G\times \{1,\dots,c_H\}$ and any $\gamma\in \sqcup_{m=1}^{c_{j,k}}\Gamma_j \alpha_{j,k,m}\Gamma_k $, $1\leq j,k\leq c_H$, one has $\dist(\Gamma_k \gamma g,\Gamma_k g)>c_4p_S^{-\kappa}$ for some constant $c_4$ unless  $\gamma=1$, because $M$ is compact and the distance $\dist$ on $M$ is locally equivalent to the distance induced by the Euclidean distance on $\M(N,\R)$, see \cite[Section 2]{RW}.
Therefore, setting $c_5=c_3+(d-1)/2$ Equation \eqref{eq:non-equiv} would imply 
\bq \label{eq:proof}
  K_{T_\alpha\circ \widetilde s_\mu} (x,x) - \delta_\alpha \cdot K_{ \widetilde s_\mu}(x,x)  =  O(\mu^{(d-1)/2}p_S^{c_5\kappa})  
\eq
provided that we prove the necessary conditions 
\begin{enumerate}
\item[(I)] If  $\alpha\in K_S$, $1$ belongs to $\sqcup_{m=1}^{c_{j,j}}\Gamma_j\alpha_{j,j,m}\Gamma_j$ for every $j$,
\item[(II)] If $\alpha\not\in K_S$, $1$ does not belong to $\sqcup_{m=1}^{c_{j,j}}\Gamma_j\alpha_{j,j,m}\Gamma_j$ for any $j$.
\end{enumerate}
The condition (I) is obvious by Lemma \ref{lem:2}, so  suppose that $1\in \sqcup_{m=1}^{c_{j,j}}\Gamma_j\alpha_{j,j,m}\Gamma_j$ for some $j$. This means that $1\in H(\Q) \cap x_j K_0\alpha K_0 x_j^{-1}$ and in particular $1\in K_0\alpha K_0$ together with $\alpha\in K_S$. Hence (II) holds by  contraposition, and  \eqref{eq:proof} is proved. Integrating this equality over  $x$  and $\mu$ we arrive at
\bq
\label{eq:24.5.2019}
\int_{-\infty}^\mu \sum_{j=0}^\infty \rho(t-\mu_j) \lambda_j(\alpha) \d t - \delta_\alpha \frac{ \varpi_d \, \vol M}{(2\pi)^d } \mu^d = O\big (\mu^{(d+1)/2}p_S^{c_5\kappa}+\mu^{d-1}\big ),
\eq
where we took into account \eqref{eq:29.5.2017},  \eqref{eq:28.3.2018},  and \eqref{eq:4}, together with  the fact that $K_{ \widetilde s_\mu}(x,y)$ is rapidly decreasing as $\mu \to -\infty$. Besides, in the present case we have $S^\ast_xM=\mklm{(x,\xi) \in T^\ast_x(M) \mid |\xi|_x =1}$, so that   
\bq
\label{eq:26.9.2019}
\int_M\vol S^\ast_x(M)\, \d x=d \int_M\vol B^\ast_x(M)\, \d x =d\,  \varpi_d \, \vol M, 
\eq
where $B^\ast_xM:=\mklm{(x,\xi) \in T^\ast_x(M) \mid |\xi|_x \leq 1}$. 
 Now, for each eigenfunction $\phi_j$ we can choose a point $y_j\in H(\Q)\bsl H(\R)/K_0$ such that $|\phi_j(y_j)|=\max_{x\in H(\Q)\bsl H(\R)/K_0}|\phi_j(x)|$, yielding the  trivial bound
\[
|\lambda_j(\alpha)|=\sum_{\beta\in K_0\alpha K_0/K_0}|\phi_j(y_j\beta)|/|\phi_j(y_j)|\leq \sharp(K_0\alpha K_0/K_0)
\]
uniformly in $j$. Furthermore, by \cite[Lemma 2.13]{ST} one has $\sharp(K_0\alpha K_0/K_0)\ll p_S^{c_6\kappa}$  for $c_6=d+N-1+\frac{1}{2}N(N+1)$. With the arguments in \cite[Proof of Corollary 2.5]{duistermaat-guillemin75} one therefore deduces for any $K>0$ the estimate  
\begin{align*}
\int_{-\infty}^\mu \sum_{j=0}^\infty \rho(t-\mu_j) \lambda_j(\alpha) \d t&=\sum_{\mu_j\leq \mu-K} \lambda_j(\alpha)  \int_{-\infty}^\infty  \rho(t-\mu_j)   \d t+O(\mu^{d-1} p_S^{c_6\kappa}).
\end{align*}
 Since  $\hat \rho(0)=\int \rho(t) \d t=1$, the assertion of the theorem follows from \eqref{eq:24.5.2019}, since $d \geq 3$. 
\end{proof}

Following Shin and Templier \cite{ST}, we  now define a certain family of automorphic representations of $H$ depending on $\mu$.
Fix an automorphic representation $\pi$ of $H$. In view of $H(\bA)=H(\R)\times H(\bA_\fin)$, one has the decompositions $\pi=\pi_\infty\otimes\pi_\fin$ and $V_\pi=V_{\pi_\infty}\otimes V_{\pi_\fin}$, where $\pi_\infty\in\widehat{H(\R)}$ and $\pi_\fin\in\widehat{H(\bA_\fin)}$, and for each eigenfunction $\phi_\infty$ in $V_{\pi_\infty}$ of $\pi_\infty(\Delta)$ we write 
\bqn
\pi_\infty(\Delta) \phi_\infty=\lambda_{\phi_\infty} \phi_\infty.
\eqn
We can then define the finite dimensional subspace
\[
V_{\pi_\infty}^{\leq \mu}:=\langle \phi_\infty \in V_{\pi_\infty} \mid \text{$\phi_\infty$ is an eigenfunction of $\pi_\infty(\Delta)$ and $\sqrt{\lambda_{\phi_\infty}}\leq \mu$}\rangle.
\]
Note that $\dim V_{\pi_\infty}^{\leq \mu}>0$ means that the Casimir eigenvalue of $\pi_\infty$ is less than or equal to $\mu^2$.
In addition, we denote the subspace of $K_0$-fixed vectors in $V_{\pi_\fin}$ by 
\[
V_{\pi_\fin}^{K_0}:=\eklm{ u\in V_{\pi_\fin} \mid \pi_\fin(k_0)u=u \;\; \forall \, k_0\in K_0}.
\]
 Now, define   $\F=\F(\mu)$ as the finite multi-set consisting of those automorphic representations $\pi\in\widehat{H(\bA)}$  with $m_\pi>0$ for which  the positive integer
\[
a_\F(\pi):=m_\pi \, \dim V^{\leq \mu}_{\pi_\infty} \, \dim V_{\pi_\fin}^{K_0}
\]
is strictly positive,  where each such $\pi$ appears in $\F$ with the multiplicity $a_\F(\pi)$. As an immediate consequence of  Theorem \ref{thm:main} we now obtain the following 
\begin{cor}[\bf Asymptotic trace formula]\label{cor:main}
In the setting of Theorem \ref{thm:main} there exists a  constant $c'>0$ such that  for each finite set $S$ of primes outside $S_0$ and each $f_S\in \mathcal{H}^\mathrm{ur}_\kappa(H(\Q_{S}))$,
\[
\sum_{\pi\in \F(\mu)} \Tr\pi_S(f_S)=f_S(1)\cdot  \frac{\vol(M) \, \varpi_d}{(2\pi)^d} \mu^d +  O(\mu^{d-1}\, p_S^{c'\kappa} \, \|f_S\|_\infty),
\]
where $\|f_S\|_\infty:=\max_{x\in H(\Q_S)} f_S(x)$.
\end{cor}
\begin{proof} To begin, note that  $\mathcal{H}^\mathrm{ur}_\kappa(H(\Q_{S}))$ is spanned by the elements $\tau_\omega:=\otimes_{p\in S} \tau_{\omega_p}$, where    $\omega=(\omega_p)_{p\in S}\in B_{S,\kappa}$ and  $B_{S,\kappa}:=\{ (\omega_p)_{p\in S}\in \prod_{p\in S} X_*(A_p) \mid \|\omega_p\|_p\leq \kappa\}$. Next, define $a_\omega:=(\omega_p(p))_{p\in S}\in\prod_{p\in S} A_p$. Then $\| a_\omega \|_S \leq \kappa$, and $\tau_\omega$ can be interpreted as  the characteristic function of $K_S a_\omega K_S$. As a consequence, $f_S\in \mathcal{H}^\mathrm{ur}_\kappa(H(\Q_{S}))$ can be written as $f_S= \sum_{\omega\in B_{S,\kappa}}f_S(a_\omega)\, \tau_\omega$, and  if $K_S\alpha K_S=K_S a_\omega K_S$,   the sum
\bq \label{eq:a}
 \sum_{\pi\in \F} \Tr\pi_S(\tau_\omega)
\eq
coincides with \eqref{eq:b}, where for $\pi=\otimes_v \pi_v$ we set $\pi_S:=\otimes_{p\in S}\pi_p$. Thus, the assertion follows from Theorem \ref{thm:main} in view of  the  bound $|B_{S,\kappa}|\leq  (c_1'\kappa)^{\mathrm{rank}_\Z X_*(A_p)}$ for some suitable $c_1'\in\N_*$. 
\end{proof}

With the preceding asymptotic trace formula, we are able to prove a Plancherel density theorem and a Sato-Tate equidistribution theorem. For this, let us first introduce the relevant measures.
Define   a counting measure on $H(\Q_S)^{\wedge,\mathrm{ur}}$ for the $S$-component of $\F$  by setting 
\bqn
\widehat{m}^\mathrm{count}_{\mu,S}:=\frac{1}{|\F|}\sum_{\pi\in \F} \delta_{\pi_S},
\eqn
where $\delta_{\pi_S}$ denotes the Dirac delta measure at $\pi_S\in H(\Q_S)^{\wedge,\mathrm{ur}}$. We then have the following  

\begin{cor}[\bf Plancherel density theorem]\label{plancherel density}
For any $f_S\in \mathcal{H}^\mathrm{ur}(H(\Q_{S}))$  we have
\[
\lim_{\mu\to\infty} \widehat{m}^\mathrm{count}_{\mu,S}(\widehat {f_S}) = \widehat{m}_S^\mathrm{Pl,ur}(\widehat {f_S}).
\]
\end{cor}
\begin{proof}
For any $f_S\in \mathcal{H}^\mathrm{ur} (H(\Q_{S}))$, we can choose a constant $\kappa>0$ such that $f_S$ is in $\mathcal{H}^\mathrm{ur}_\kappa (H(\Q_{S}))$.
Corollary \ref{cor:main} implies that 
\[
\frac{(2\pi)^d} {\mu^d \, \vol(M) \, \varpi_d } \sum_{\pi\in \F(\mu)} \widehat{f_S}(\pi_S) = f_S(1) + O(\mu^{-1} p_S^{c' \kappa} \|f_S\|_\infty).
\]
Since integration of   \eqref{eq:28.3.2018} over $x$ and $\mu$ yields Weyl's law $|\F(\mu)|=\# \mklm{j\mid \mu_j \leq \mu}= \frac{\mu^d \, \vol(M) \, \varpi_d }{(2\pi)^d} \mu^d+O(\mu^{d-1})$, the assertion follows by taking the limit $\mu \to \infty$ in the last equality for each $\kappa$ separately. 
\end{proof}

\begin{cor}[\bf Sato-Tate equidistribution theorem]\label{Sato-Tate equiv}
Fix  $\theta\in\mathscr{C}(\Gamma_1)$, and let $\widehat{f}$ be a continuous function on $\widehat{T}_{c,\theta}/ \Omega_{c,\theta}$. By \eqref{eq:ST(5.2)},
 $\widehat{f}$ can be extended to a continuous function $\widehat{f}_p$ on $G_p^{\wedge,\mathrm{ur,temp}}$ for any $p\in \mathcal{V}(\theta)$. 
Let $\{(p_k,\mu_k)\}_{k\geq 1}$ be a sequence in $\mathcal{V}(\theta)\times\R_{>0}$ such that $p_k\to \infty$ and $p_k^l / \mu_k \to 0$ as $k\to \infty$ for any integer $l\geq 1$.
Then
\[
\lim_{k\to\infty} \widehat{m}^\mathrm{count}_{\mu_k,p_k}(\widehat {f}_{p_k}) = \widehat{m}^\mathrm{ST}(\widehat {f})
\]
where we wrote  $\widehat{m}^\mathrm{count}_{\mu_k,p_k}$ for $\widehat{m}^\mathrm{count}_{\mu_k,\{p_k\}}$.
\end{cor}
\begin{proof}
To begin, notice that there exists a constant $\kappa>0$ such that $f_{p_k}$ belongs to  $\mathcal{H}^\mathrm{ur}_\kappa (H(\Q_{p_k}))$ for any $k$, where $f_{p_k}$ denotes the inverse image of $\widehat {f}_{p_k}$.
Now, by Corollary \ref{cor:main} we have
\[
\frac{(2\pi)^d} {\mu^d \, \vol(M) \, \varpi_d } \sum_{\pi\in \F(\mu_k)} \widehat{f_{p_k}}(\pi_{p_k}) =  \widehat{m}_{p_k}^\mathrm{Pl,ur}(\widehat {f}_{p_k})  + O(\mu_k^{-1} p_k^{c' \kappa} \|f_{p_k}\|_\infty).
\]
Since $\|f_{p_k}\|_\infty$ does not depend on $p_k$, and by assumption we have  $\mu_k^{-1} p_k^{c' \kappa}\to 0$ as $k\to\infty$, the assertion is proved by using the same argument than in the proof of Corollary \ref{plancherel density}.
\end{proof}

\section{Equivariant asymptotics  for Hecke eigenvalues and Sato-Tate equidistribution}\label{sec:5} 

Let us now turn to the equivariant situation. To begin, we collect some basic facts about orbital integrals needed in the sequel. 

\subsection{Orbital integrals}\label{orbital}  Choose $\Theta: g \mapsto \,^t g^{-1}$ as Cartan involution on $G$, and suppose as we may that $K=G\, \cap\, \SO(N)$, where $K$ denotes a maximal compact subgroup of $G$. 
Let $T$ denote a Cartan subgroup in $G$, and  suppose that $T$ is $\Theta$-stable. 
Notice that any semisimple element is conjugate to an element of a $\Theta$-stable Cartan subgroup in $G$, see \cite[Theorem 5.22]{knapp}.
For each $\gamma\in T$, let $G_\gamma$ denote the centralizer of $\gamma$ in $G$, and  $\g_\gamma:=\{ X\in\g \mid \mathrm{Ad}(\gamma)X=X  \}$ its Lie algebra. We then introduce the \emph{orbital integral}
\[
J(\gamma,f):=J^{G/T}(\gamma,f):=|D(\gamma)|^{1/2} \int_{G_\gamma\bsl G} f(g^{-1}\gamma g) \, \d g,  \qquad \gamma\in T, \;\; f\in C_c^\infty(G), 
\]
where  $D(\gamma):=D^G(\gamma):=\det((1-\mathrm{Ad}(\gamma))|_{\g/\g_\gamma})$ is  the Weyl discriminant. Denote the Lie algebra of $T$ by $\t$, and write $\g_\C$ and $\t_\C$ for the respective complexifications.
It is well-known that $J(\gamma,f)$ defines a compactly supported\footnote{Here compactness is to be understood with respect to  the relative topology on $T'$ induced by $T$.} smooth function on the subset  $T'\subset T$ of regular elements of $T$, see \cite[Propositions 11.7]{knapp}. Since the structure of a single $J(\gamma,f)$ is rather involved, it is convenient to consider superpositions of orbital integrals of the following form. Let $W_I$ denote the Weyl group generated by reflections corresponding to the imaginary roots in $(\g_\C,\t_\C)$. One then defines the \emph{stable orbital integral} 
\[
\bar{J}(\gamma,f):=\bar{J}^{G/T}(\gamma,f):=\sum_{w\in W_I} J(w\gamma,f) .
\]

To describe the structure of the stable orbital integrals more explicitly,   let $\fa$ be a maximal Abelian subspace in $\p$ with respect to the Cartan decomposition \eqref{eq:cartan}, and put $A:=\mathrm{exp}(\fa)$. Consider  the corresponding Iwasawa decomposition 
\bq
\label{eq:Iwasawa}
G=A\, UK
\eq
 of $G$, $U$ being a unipotent subgroup $U$ in $G$. There is an algebra isomorphism $\CT(K\bsl G/ K)\ni f \to \A f \in \CT(\fa/W)$ called the  \emph{Abel transform} given by
\bq
\label{eq:Abeltrans}
(\A f)(X):=e^{\rho(X)} \int_U f(\exp(X)n) \, \d n, 
\eq
where $W$ is the Weyl group of $(\g_\C,\fa_\C)$ and $\rho$ denotes the half sum of positive roots of $(A,U)$. We may suppose that $T=T_AT_K$, where $T_K:=T\cap K$ and $T_A:=T\cap A$. Further, for any integrable function $h \in L^1(\fa)$ let 
$$
\widehat h(\lambda):=\int_\fa h(X) e^{\lambda(X)} \d X, \qquad \lambda\in i\fa^*, 
$$
denote its Fourier transform.  Let  $\gamma\in T$ be arbitrary, and without loss of generality suppose that $\fa\cap\g_\gamma$ is a maximal Abelian subspace in $\p\cap\g_\gamma$. Let $B$ denote a $\Theta$-stable Cartan subgroup of $G$ containing $A$, and $\fb$ the Lie algebra of $B$.
Set $\fb_\gamma:=\fb \cap \g_\gamma$. 
The following proposition is a consequence of Herb's Fourier inversion formula \cite{Herb2,Herb1} for $\bar{J}^{G/T}(\gamma,f)$, which expresses the latter in terms of the Fourier transform $\widehat{\A f}$ of $\A f$. 
\begin{proposition}
If $f\in C_c^\infty(K\bsl G/K)$  is non-negative one has 
\bq
\label{eq:28.12.2019}
J(\gamma,f)\leq \bar{J}(\gamma,f) \ll \int_{i\fa^*} \, \big| \widehat{\A f}(\lambda) \big|\, (1+\|\lambda\|)^{\frac{r_\mathrm{nc}(\gamma)}{2}} \, \d \lambda ,
\eq
where $r_\mathrm{nc}(\gamma)$ is the number of non-compact roots in $(\g_{\gamma,\C},\fb_{\gamma,\C})$. 
\end{proposition}
\begin{proof} 
First, consider the case where $\gamma\in T'$.
Let  $\mathcal{M}$ denote the centralizer of $T_A$ in $G$. Clearly, $A \subset \mathcal{M}$. Since $T$ is commutative and $T_A \subset T$, we have $T \subset \mathcal{M}$ as well, and by \cite[Proposition 4.7]{DKV1} or \cite[(11.42)]{knapp} one has
\[
  \bar{J}^{G/T}(\gamma,f) =  \bar{J}^{\mathcal{M}/T}(\gamma, f_{\mathcal{U}}), 
\]
where $\mathcal{U}$ denotes the unipotent radical of the parabolic subgroup $\mathcal{M}U$, and for $m \in \mathcal{M}$ we set $f_{\mathcal{U}}(m):= \eta(m) \, \int_{\mathcal{U}}f(m u)\, \d u$,  $\eta$ being a non-negative real-valued quasi-character on $\mathcal{M}$.
Notice that $\eta$ is trivial on $T_K$, and that $W_I$ does not act on $T_A$. Now, since $G=H(\R)$ and $H$ is connected, there exists an algebraic torus $\mathcal{T}$ over $\R$ such  that $\gamma  \in \mathcal{T}(\R)$ and $\mathcal{T}(\R)\subset T$, compare \cite[Corollary 13.3.8 (i)]{springer}. It is known that $\mathcal{T}(\R)$ is isomorphic to $(\R^\times)^{n_1}\times (\R_{>0})^{n_2}\times (\C^1)^{n_3}$ for some $n_1$, $n_2$, $n_3\in\N$, where $\C^1:=\{z\in \C \mid |z|=1\}$.
The part $(\R^\times)^{n_1}\times (\R_{>0})^{n_2}$ is included in the center $C(\mathcal{M})$ of $\mathcal{M}$.
Since the Fourier transform on $C(\mathcal{M})$ is obvious, the problem is reduced to the case $J^{\mathcal{M}/T}(\gamma', f_{\mathcal{U}})$ where $\gamma'$ denotes the  $(\C^1)^{n_3}$-part  of $\gamma$.
Note that $\mathrm{rank}\, \mathcal{M}=\mathrm{rank}\, \mathcal{K}$, where $\mathcal{K}:=\mathcal{M}\cap K$.
Since $(\C^1)^{n_3}$ is connected, $\gamma'$ belongs to the connected component $\mathcal{M}^0$ of the identity in $\mathcal{M}$.
Therefore, we may assume that $\mathcal{M}$ is connected in view of the equality $J^{\mathcal{M}/T}(\gamma',f)=\frac{[\mathcal{M}\, :\, \mathcal{M}^0]}{[\mathcal{M}_{\gamma'}\, :\, \mathcal{M}_{\gamma'}\cap \mathcal{M}^0]}J^{\mathcal{M}^0/(T\cap \mathcal{M}^0)}(\gamma',f)$.
By the above mentioned conditions on $\mathcal{M}$ and $\mathcal{K}$, we can now apply Herb's Fourier inversion formula \cite[Theorem 1]{Herb1}, \cite[Theorem 2]{Herb2} to $J^{\mathcal{M}/T}(\gamma',f)$, and consequently obtain an explicit formula for  $\bar J^{G/T}(\gamma,f)$. Without explaining the details, we obtain as a result for any $f \in C_c^\infty(K\bsl G/K)$ and regular  $\gamma \in T'$ the expression 
\bq  \label{eq:20191213}
\bar{J}^{G/T} (\gamma,f)= \int_{i\fa^*}  \widehat{\A f}(\lambda) \, \Phi(\gamma,\lambda) \, \d \lambda, 
\eq
where $\Phi(\gamma,\lambda)$ is an explicitly given smooth function on $ T'\times i\fa^*$. To give a closer description of $\Phi(\gamma,\lambda)$, let $\Delta$ denote a root system in $(\g_\C,\fb_\C)$, and $\Delta_\R$ the subset of real roots in $\Delta$. 
For each $\alpha\in\Delta$, define an element $\bar{H}_\alpha$ in $\fb$ by $\alpha(H)=\langle H,\bar{H}_\alpha\rangle$ for all $H\in\fb$, and set $H_\alpha:=2\bar{H}_\alpha/\langle\bar{H}_\alpha,\bar{H}_\alpha\rangle\in\fb$.
Let $\fa_\mathcal{M}$ denote the Lie algebra of $A\cap C(\mathcal{M})$, so that we have the orthogonal direct sum $\fa=\fa_\mathcal{M}\oplus\fa^\mathcal{M}$. 
Take strongly orthogonal real roots $\alpha_1$, $\alpha_2,\dots,\alpha_t$ in $\{ \alpha\in\Delta_\R \mid H_\alpha \in \fa^\mathcal{M} \}$ following \cite[Section 2]{Herb2}, and take elements $\alpha_{t+1},\dots,\alpha_s\in \fa^*$ such that $\alpha_1,\dots,\alpha_s$ form a basis in $\fa^*$. Set $H_j:=H_{\alpha_j}(\in\fa^\mathcal{M})$ for $1\leq j\leq t$, and  for $t<j\leq s$  denote the orthogonal projection of $H_{\alpha_j}$ to $\fa_\mathcal{M}$ by $H_j$. 
Let $\mathfrak{m}$ denote the Lie algebra of $\mathcal{M}$, take an element $y$ in $\mathrm{Ad}(\mathfrak{m}_\C)$ such that $y^{-1}\t_\C y=\fb_\C$, and denote the Weyl group of $(\mathfrak{m}_\C,\t_\C)$ by $W^{\mathfrak{m}_\C}(\subset \mathrm{Ad}(\mathfrak{m}_\C))$.  
Then, by Herb's formula \cite{Herb2,Herb1} one has 
\bq\label{eq:9.11.2020}
\Phi(\gamma,\lambda)= \sum_{w\in W^{\mathfrak{m}_\C}}\Phi_{y,w}(\gamma,\lambda),
\eq
each  $\Phi_{y,w}(\gamma,\lambda)$ being given  in terms of  a linear combination of products of the functions
\bq \label{eq:factors}
e^{\pm \theta_j(\gamma) \lambda(H_j) } , \quad e^{ \pi \lambda(H_j) } , \quad \sinh(\pi \lambda(H_j))^{-1} , \quad \sinh(\pi \lambda(H_j+H_{j+1}))^{-1},
\eq
where $ \gamma\in T', \; \lambda\in i\fa^*, \; 1\leq j\leq s$,  and $\theta_j(\gamma)\in\R$ is determined by the condition
\[
y^{-1}w^{-1}\gamma wy=h_0\exp\Big(i\sum_{j=1}^t \theta_j(\gamma) H_j + \sum_{j=t+1}^s \theta_j(\gamma) H_j\Big )\in \exp(\fb_\C)
\]
for some $h_0\in B\cap K$. In particular, one sees that $\Phi(\gamma,\lambda)$ is uniformly bounded in $\norm{\lambda}$ from above, yielding \eqref{eq:28.12.2019} in this case. Next, let us consider the case where $\gamma_0\in T\setminus T'$ is a singular element. For each $y\in G_{\gamma_0}$, set
\[
f_{\gamma_0}(y):=\frac{|D^G(\gamma_0 y)|^{1/2}}{|D^{G_{\gamma_0}}(y)|^{1/2}}\int_{G_{\gamma_0}\bsl G} f(x^{-1}\gamma_0 y x)\, \d x,
\]
and fix a chamber  $\mathfrak{c}:=\{ H\in\t\mid \alpha(H)>0 \, \text{ for all } \alpha\in\Delta_{\t_\C, +} \}$ with respect to a positive root system $\Delta_{\t_\C, +}$ in $(\g_\C,\t_\C)$.  
Let $\Delta_{\gamma_0}$ denote a positive root system in $(\g_{{\gamma_0},\C},\t_\C)$ and for $\alpha \in \Delta_{\gamma_0}$ write $D_\alpha$  for the invariant differential operator on $T$ corresponding to  $H_\alpha\in \t$, that is, $D_\alpha \Phi(\delta)=\lim_{\theta\to 0} \frac{\d}{\d \theta} \Phi(\delta\exp(\theta H_\alpha))$ where $\Phi$ is a differentiable function on $T'$.
Then, by Harish-Chandra's limit formula \cite[Theorem 4]{HC1957},
\[
J^{G/T}({\gamma_0},f)=f_{{\gamma_0}}(1)=\lim_{\delta\to 1,\, \delta\in \exp(\mathfrak{c})} D_{\gamma_0} J^{G_{\gamma_0}/T}(\delta,f_{{\gamma_0}}) =\lim_{\delta\to 1,\, \delta\in \exp(\mathfrak{c})} D_{\gamma_0} J^{G/T}({\gamma_0}\delta,f) 
\]
where  $D_{\gamma_0}:=c_{\gamma_0}\prod_{\alpha\in\Delta_{\gamma_0}}D_\alpha$ acts on the variable $\delta\in \exp(\mathfrak{c})$ and  $c_{\gamma_0}$ is some constant. Since $D_{\gamma_0}$ remains unchanged under the action of $W_I$ on $\delta$, we conclude that 
\begin{align}
\begin{split}
\label{eq:20191213a}
J({\gamma_0},f)\leq \bar{J}({\gamma_0},f)&=\sum_{w\in W_I}\lim_{\delta\to 1,\, \delta\in \exp(\mathfrak{c})} D_{\gamma_0} J((w{\gamma_0})\delta,f) = \lim_{\delta\to 1,\, \delta\in \exp(\mathfrak{c})} D_{\gamma_0} \bar{J}({\gamma_0}\delta,f)
\end{split}
\end{align}
with $\bar{J}({\gamma_0}\delta,f)$ given by \eqref{eq:20191213}. 
For each $w\in W(\g_\C,\t_\C)$ \textcolor{black}{and  $\alpha \in \Delta_{\gamma_0}$, the differential operator $D_\alpha$} acts on the factors $e^{\pm \theta_j(\gamma_0\delta) \lambda(H_j) }$ only if $\alpha^{yw}\in \Delta$ is not compact, where $\alpha^{yw}(H):=\alpha(y^{-1}w^{-1}Hwy)$.
Hence, \eqref{eq:9.11.2020} implies that {$D_{\gamma_0} \Phi(\gamma_0\delta,\lambda)$} consists of  a linear combination of products whose  factors  are given by the expressions \eqref{eq:factors} and polynomials in $\lambda$ whose degrees are less than or equal to $r_\mathrm{nc}(\gamma)/2$. 
Therefore, {$D_{\gamma_0} \Phi(\gamma_0\delta,\lambda)$} is uniformly bounded from above by $(1+\|\lambda\|)^{r_\mathrm{nc}(\gamma)/2}$ for any $\delta\in \exp(\mathfrak{c})$, and the assertion follows.
\end{proof}

\subsection{Equivariant asymptotics and equidistribution results}\label{Equivariant case}

We are now ready to derive asymptotics for Hecke eigenvalues in the equivariant setting. With the notation as in Section \ref{sec:4}, let $K$ be a maximal compact subgroup of $G=H(\R)$, so that  $C(G)\subset K$.
Further, we may suppose that $G$ is not compact.
Denote by $Z_H$  the center of $H$, and set
\[
Z:=Z_H(\Q)\cap K_0.
\]
Clearly,  $Z\subset C(G)=Z_H(\R)$. 
Choose an irreducible representation $\sigma$ in $\widehat K$.
It is obvious that $L^2_\sigma(H(\Q)\bsl H(\bA)/K_0)=0$ if $\sigma$ is not trivial on $Z$.
Hence, we may suppose that $\sigma$ is trivial on $Z$, so  that 
\[
Z_\sigma:=\mathrm{Ker}(\sigma)\supset Z.
\]
Notice that   $C(\Gamma_j)=\Gamma_j\cap Z_H(\Q) = Z$ for any $j$, where $C(\Gamma_j)$ denotes the center of $\Gamma_j$, since $\Gamma_j$ is Zariski dense in $H$, see \cite[Theorem 4.10]{PR}. 
 To begin, we need the following  variant of Lemma \ref{lem:15.04.2018}. 
\begin{lem}\label{lem:21.04.2018}  
Let $T_\alpha$ be a Hecke operator as in Lemma \ref{lem:relation}.  For  $\eps \geq 0$, denote by $f_\eps: G \rightarrow \mklm{0,1}$ the characteristic function of 
$K_\eps:=\{ g\in G \mid \dist_G(K,gK)\leq \eps   \}$. 
Then,  one  has for each  $x=(g,l) \in M$ and $\eps>0$  the asymptotic formula
\begin{align*}
&  K_{T_\alpha\circ \widetilde s_\mu\circ \Pi_\sigma} (x,x) -  \left [ \sum_{k \in  \Gamma(\alpha,l)}\sigma(k) +\sum_{y \in T(\alpha,x)-C(\alpha,x)}\sigma(k_y)\right ] K_{\widetilde s_\mu\circ \Pi_\sigma} (x,x) \\
& =  O\Big( \mu^{d-\dim K-1} \sum_{m=1}^{c_{l,l}}  \sum_{\beta\in \Gamma_l\bsl \Gamma_l\alpha_{l,l,m}\Gamma_l} ( f_\eps(g^{-1}\beta g)-f_0(g^{-1}\beta g) )  \Big) \\
& \quad +O\Big (       (\mu/\eps)^{(d-\dim K -1)/2}  \sum_{m=1}^{c_{l,l}} |\Gamma_l\bsl \Gamma_l\alpha_{l,l,m}\Gamma_l|  +\mu^{-\infty}\,   \sum_{j=1}^{c_H} \sum_{m=1}^{c_{j,l}} |\Gamma_j\bsl \Gamma_j\alpha_{j,l,m}\Gamma_l|\Big ).
\end{align*}
\end{lem}
\begin{proof}
Let $x=(g,l)$ be fixed and $y=(\beta g,l) \in H(\alpha,x)$. Then
\begin{align*}
\dist(xK,yK) =\dist (\Gamma_lgK , \Gamma_l\beta gK)
=\inf_{\gamma \in \Gamma_l} \dist_G (K, g^{-1}\gamma\beta g K).
\end{align*}
Assuming as we may that $\beta\equiv \Gamma_l \beta$ has been chosen such that  $\inf_{\gamma \in \Gamma_l} \dist_G (K, g^{-1}\gamma\beta g K)$ is attained by $\dist_G (K, g^{-1}\beta g K)$ we obtain 
\bqn 
\dist(xK,yK) \leq \eps \qquad \Longleftrightarrow \qquad f_\eps(g^{-1}\beta g)=1.
\eqn
Furthermore, $f_0(g^{-1}\beta g)=1$ iff $y \in T(\alpha,x)$. Consequently, for $\dist(xK,yK) \leq \eps$ we have 
\bqn 
f_\eps(g^{-1}\beta g)-f_0(g^{-1}\beta g)= \begin{cases} 1 & \text{iff } y \in H(\alpha,x) -T(\alpha,x),  \\ 0  & \text{iff } y \in T(\alpha,x). \end{cases}
\eqn
Thus,
\begin{gather*}
 K_{T_\alpha\circ \widetilde s_\mu\circ \Pi_\sigma} (x,x) -  \sum_{y \in   T(\alpha,x)}K_{\widetilde s_\mu\circ \Pi_\sigma} (y,x)=  \sum_{y \in   H(\alpha,x)-T(\alpha,x)}  K_{\widetilde s_\mu\circ \Pi_\sigma} (y,x) \\ 
 = \sum_{y \in   H(\alpha,x)} \big ( f_\eps(g^{-1}\beta g)-f_0(g^{-1}\beta g) \big ) K_{\widetilde s_\mu\circ \Pi_\sigma} (y,x)  + \sum_{y \in   H(\alpha,x)-T(\alpha,x), \, \dist(xK,yK) > \eps }  K_{\widetilde s_\mu\circ \Pi_\sigma} (y,x)
\end{gather*}
up to terms of order $O(\mu^{-\infty})$ times the cardinality of the sum in \eqref{eq:4cis}. The assertion now follows from Proposition \ref{prop:30.01.2017} along the lines of the proof of Lemma \ref{lem:15.04.2018} by taking into account \cite[Remark 3.4]{RW}.
\end{proof}

To proceed, we need the following{\footnote{Here is some overlap with recent results in \cite[Lemma 7.11]{brumley-marshall20}, though our proof is independent and methodologically different.}
\begin{proposition}\label{lem:20190605}
Let $\widetilde K$ denote the maximal compact normal subgroup of $G$, that is, the product of the center $C(G)$ and all compact simple factors of $G$.
Fix $m\in \N$, and let $\beta\in H(\Q)\cap \M(N,\frac{1}{m}\Z)$ be such that $\beta \notin \widetilde K$. Choose a bounded domain $D$ in $G$.
Then, for any  $0< \eps\ll \log(1+1/Nm^2)$ and any $0<s<1$ we have 
\[
\int_D f_\eps (g^{-1}\beta g)\, d g \ll_{D}  m^{N}  \, s^{-1} \, \eps^{1-s}.
\]
\end{proposition}
\begin{proof}
Recall the notations and the setting in Section \ref{orbital}.
To begin, we define an inner product  on $\M(N,\R)$ by setting $(X,Y):=\Tr(X\, {}^t\!Y)$ and a  norm $\|X\|:=(X,X)^{1/2}$.
The corresponding distance is  locally equivalent to the distance  $\dist_G$ on $G\subset \M(N,\R)$. 
Denote by $\Delta$ the root system of $(\g,\fa)$, by $\Delta_0$ the set of positive simple roots, and by $W$ the corresponding Weyl group. 
Further, recall the polar decomposition $G=KAK$, by which every $g\in G$ can be written as $g=k_1 \cdot \exp(X(g)) \cdot k_2$ where $k_i \in K$, and $X(g) \in \fa$ is uniquely determined up to conjugation by  $W$. Introducing the positive Weyl chamber $\fa^+:=\mklm{X \in \fa \mid \alpha(X)>0 \, \forall \, \alpha \in \Delta_0}$, this decomposition induces a mapping  $X:G\to \overline{\fa^+}$ such that
\bqn
\|X(g)\|\ll \dist_G(K,gK)\ll \|X(g)\|
\eqn
uniformly in  $g\in D$. Further, by \cite[Lemma 4.2]{MT},
\bq 
\label{eq:MT}
\|X(g)\|^2 \ll_D\mathfrak{L}(g):=\log\big ( \norm{g}^2/N\big ) \leq 2\|X(g)\|, \qquad g \in D, 
\eq
 with $\mathfrak{L}(g)=0$ iff $g \in K$. Now, let $\beta \in G$ be arbitrary.  
 By the $K$-bi-invariance of $f_\eps$ one computes with respect to the global Cartan decomposition \eqref{eq:globalCartan}
 \[
\int_D f_\eps (g^{-1}\beta g)\, d g= \int_K \int_{D_\p} f_\eps (\exp(-X) \cdot \beta \cdot \exp X)\, d X \d k \ll  \int_{D_\p} f_\eps (\exp(-X) \cdot \beta \cdot \exp X)\, d X, 
\]
where $D_\p\subset \p$ is a bounded domain and $dX$ a suitable measure on $\p$. Let us examin the last integral more closely by  introducing  the \emph{$\beta$-displacement function} 
\bqn 
\delta_\beta(X):= \dist_G(K, \exp(-X) \cdot \beta \cdot \exp X \cdot K), \qquad X \in \p,
\eqn
which can also be regarded as a function on the Riemannian symmetric space $G/K$ in view of the diffeomorphism $G/K\simeq \p$.
If $\beta$ is semisimple,  the infimum of $\delta_\beta$ is reached, and the points where it is reached constitute a  submanifold  $S_\beta\subset \p$,  see  \cite[p. 279]{helgason78} and \cite[Proposition 5.7]{DKV1}. Furthermore, the minimum of $\delta_\beta$ is given by  $\| X_\beta \|$ if one writes $\beta =\exp(X_\beta) \cdot  k_\beta$ with respect to the  decomposition \eqref{eq:globalCartan}. Notice that since  $H(\Q)\bsl H(\bA)$ is compact, all elements in $H(\Q)$ are semisimple.\footnote{This fact  is crucial for the following.} 

Now, assume that $\beta \in H(\Q)\cap  \M\big (N,\frac{1}{m}\Z\big )$ for some $m \in \N$, but $\beta\notin K$, so that $X_{\beta}\not=0$. By the above, 
\bqn 
f_\eps( \exp(-X) \cdot \beta \cdot \exp X) =0 \qquad \text{for all $X \in \p$ if } \, \eps< \norm{X_{\beta}},
\eqn
and by \eqref{eq:MT} we have $\mathfrak{L}(\beta)\leq 2\|X(\beta)\| \leq 2 \norm{X_{\beta}}$, while
\bqn 
\frac{\Tr (\beta^t \beta)}N=\frac L{Nm^2}>1, \qquad L \in \mklm{Nm^2+1, Nm^2+2, \dots }.
\eqn
Consequently, we conclude for  all $X \in \p$ that 
\bqn 
f_\eps( \exp(-X) \cdot \beta \cdot \exp X) =0 \qquad \text{if } \, \eps\ll  \log \Big (1 +\frac1{Nm^2}\Big ),
\eqn
yielding the assertion for $\beta \notin K$. 

Next, let us suppose that $\beta \in H(\Q)\cap  \M\big (N,\frac{1}{m}\Z\big )\cap K$, so that  $\inf \delta_\beta=0$. Our intention is to make use of the upper bound \eqref{eq:28.12.2019}  for orbital integrals to show the desired estimate in this case. For this sake notice that since $H$ is a closed subgroup of $\SL(N)$ over $\Q$, the Lie algebra $\mathfrak{h}\subset \M(N,\Q)$ of $H$ is a $\Q$-vector space, so that   $\mathfrak{h}\otimes \R\simeq \g$ and  $\mathfrak{h}=\g\cap \M(N,\Q)$. Further,  $\beta\in H(\Q)$ implies that  the $\R$-subspace $\g_\beta:=\{ X\in\g \mid \beta X=X\beta \}$ has a basis $\{Y_j \}_{1\leq j\leq \dim \g_\beta}$ consisting of matrices  $Y_j\in \M(N,\Z)$.
Consequently,  $L:=\g_\beta^\perp\cap \M(N,\Z)$ must be a $\Z$-lattice in the orthogonal complement $\g_\beta^\perp:=\{ X\in\g \mid (X,Y_j)=0,  \;\; 1\leq j\leq \dim\g_\beta\}$. 
In addition, $\g_\beta^\perp$ is $\mathrm{Ad}(\beta)$-stable since $(\beta X\beta^{-1},Y)=( X,\beta Y\beta^{-1})$ for any  $\beta\in K$.
Thus,  $m^2 \beta L {}^t\!\beta\subset L$, and we conclude that 
\[
|D(\beta)|:=\big |\det((\1-\mathrm{Ad}(\beta))|_{\g_\beta^\perp})\big |\in \frac{1}{m^{2N}}\N.
\]
In view of   $\beta\not\in C(G)$, this implies that  $1\leq m^{N} |D(\beta)|^{1/2}$. Now, choose a $\Theta$-stable Cartan subgroup $T$ in $G$ such that $T_K:=T\cap K$ is a maximal torus in $K$. 
There exists an element $k_0\in K$ such that $k_0^{-1}\beta k_0=\beta_0\in T_K$, and without loss of generality we may suppose that $D$ is left $K$-invariant. Since  $|D(\beta)|=|D(\beta_0)|$, we obtain
\[
\int_D f_\eps(g^{-1}\beta g) \, \d g=\int_D f_\eps(g^{-1}\beta_0 g) \, \d g \leq m^{N} \, |D(\beta_0)|^{1/2}\int_D f_\eps(g^{-1}\beta_0 g) \, \d g.
\]
 Further, there are only finitely many possibilities for  centralizers of elements in $T_K$, so that normalizing their Haar measures we arrive at
\bq
\label{eq:31.12.2019}
\int_D f_\eps(g^{-1}\beta g) \, \d g \ll m^{N} \, |D(\beta_0)|^{1/2}\int_{D_{\beta_0}\times D'}f_\eps(g^{-1}\beta_0 g) \, \d g \ll  m^{N} \,  J^{G/T}(\beta_0,f_\eps), 
\eq
where we wrote $D_{\beta_0}\subset G_{\beta_0}$ and $D'\subset G_{\beta_0}\bsl G$ for the projections of $D$  with respect to the decomposition $G\simeq G_{\beta_0}\times G_{\beta_0}\bsl G$. 

In order to use the upper bound \eqref{eq:28.12.2019}, we have to replace $f_\eps$ by a  test function $\widetilde f_{\eps} \in \CT( K \bsl G /K)$ in a suitable way. Recall that  $f_\eps$ is  the characteristic function of the $K$-bi-invariant compact set $K_\eps:=\mklm{g \in G  \mid \dist_G(K,gK) \leq \eps}$. Using standard techniques  one can construct   a  function $\widetilde f_\eps\in \CT(K \bsl G/K)$ which,  say,  equals $1$ on  $K_{2\eps}$ and is supported inside $K_{4\eps}$,  compare \cite[Theorem 1.4.1]{hoermanderI}.  Furthermore, one can achieve  that the restriction of  $\widetilde f_\eps$ to $A \, U\simeq \fa \times \mathfrak{u}\simeq \p$ with respect to the Iwasawa decomposition \eqref{eq:Iwasawa} 
is essentially of the form
\bqn 
\widetilde f_{\eps}\equiv  f_{3\eps} \ast \chi_\eps,
\eqn
where $\chi \in \CT(\fa \times \mathfrak{u})$ denotes a non-negative function with support in the unit ball,  $\int_{\fa \times \mathfrak{u}} \chi =1$, and $\chi_\eps:=\eps^{-\dim (A\times U)}\chi(\cdot/\eps)$. Furthermore, writing the integral over $U$ in \eqref{eq:Abeltrans} as an integral over its Lie algebra $\mathfrak{u}$ one computes 
\begin{align}
\begin{split}
\label{eq:27.12.2019}
\widehat{\A \widetilde f_{\eps}}(\lambda)&=\int_{\fa}\bigg ( \int_ {\mathfrak{u}} e^{\rho(X)} \underbrace{\widetilde f_{\eps}(e^Xe^Y )}_{\equiv (f_{3\eps} \ast \chi_\eps)(X,Y)} \d Y  \bigg ) e^{\lambda(X)}  \d X \\ 
&\equiv \int_{\fa\times \mathfrak{u}}  \left ( \int_{\fa \times \mathfrak{u}} f_{3\eps}(X',Y')\chi_\eps(X-X',Y-Y') \d Y' \d X' \right )  e^{(\lambda+\rho)(X)} \d Y  \d X \\ 
&= \int_{\fa \times \mathfrak{u}} \left ( \int_{\fa \times \mathfrak{u}} f_{3\eps}(\eps X',\eps Y')\chi(X/\eps-X',Y/\eps-Y') \d Y' \d X' \right )  e^{(\lambda+\rho)(X)} \d Y  \d X\\
&= \eps^{\dim (A\times  U)} \int_{\fa} e^{\eps \lambda(X)} \underbrace{ \left (e^{\eps \rho (X)} \int_{\mathfrak{u}}  \int_{\fa \times \mathfrak{u}} f_{3\eps}(\eps X',\eps Y')\chi(X-X',Y-Y') \d Y' \d X'  \d Y \right )}_{=:\mathcal{B}_\eps(X) \in \CT(\fa)}      \d X\\
&=: \eps^{\dim (A\times  U)} \widehat{\mathcal{B}_{\eps}}(\eps \lambda),
\end{split}
\end{align}
where $\widehat{\mathcal{B}_{\eps}}(\lambda) \in \S(i\fa^\ast)$ is rapidly decreasing in $\lambda$ uniformly in $\eps$.  Now, by assumption,  $\beta_0 \notin \widetilde K$, which  implies that
\bq
\label{eq:30.12.2019}
\text{ the root system of $(\g_{\beta_0,\C},\fb_{\beta_0,\C})$ does not contain all non-compact roots in $(\g_\C,\fb_{\C})$.}
\eq
In fact, let $G_1$ be a non-compact simple linear algebraic $\R$-group, let $A_1$ denote the $\R$-connected component of the identity in a maximal split algebraic $\R$-torus in $G_1$, and $B_1$ a Cartan subgroup of $G_1$ containing $A_1$.
Set $\g_1:=\mathrm{Lie}(G_1)$, $\fa_1:=\mathrm{Lie}(A_1)$, and $\fb_1:=\mathrm{Lie}(B_1)$.
Suppose that a semisimple element $\gamma_1$ in $G_1$ commutes with all root spaces of $(\g_1,\fa_1)$.
Then $\gamma_1$ commutes with $A_1$. It is clear that any non-compact root space in $(\g_{1,\C},\fb_{1,\C})$ is a subspace of a root space of $(\g_{1,\C},\fa_{1,\C})$.
Since by $\mathfrak{sl}_2$-triple theory every non-trivial nilpotent element has a non-trivial factor in a root space of $A_1$, $\gamma_1$ commutes with all unipotent elements. The Bruhat decomposition then implies  that $\gamma_1\in C(G_1)$ because $G_1$ is simple, yielding \eqref{eq:30.12.2019}. 

In view of \eqref{eq:30.12.2019} we can now apply the bound \eqref{eq:28.12.2019} to estimate $J^{G/T}(\beta_0,\widetilde f_{\eps})$,  and with \eqref{eq:27.12.2019} we obtain  for sufficiently large  $s'>0$ the estimate 
\begin{align*}
J^{G/T}(\beta_0,f_\eps)& \leq J^{G/T}(\beta_0,\widetilde f_{\eps}) \ll \int_{i\fa^*} \, \big| \widehat{\A \widetilde f_{\eps}}(\lambda) \big|\, (1+\|\lambda\|)^{\dim U-1} \, \d \lambda\\
&=\eps^{\dim U+\dim A} \int_{i\fa^*} \, \big|  \widehat{\mathcal{B}_{\eps}}(\eps \lambda) \big|\, (1+\|\lambda\|)^{\dim U-1} \, \d \lambda \\
&=\eps^{\dim U+\dim A-s'} \int_{i\fa^*} \, \big|  \widehat{\mathcal{B}_{\eps}}(\eps \lambda) \big|\,(\eps+\|\eps \lambda\|)^{s'} (1+\|\lambda\|)^{\dim U-1-s'} \, \d \lambda\\ 
&=\eps^{\dim U+\dim A-s'} \sup_{\lambda \in i \fa^\ast} \big|  \widehat{\mathcal{B}_{\eps}}(\lambda) (1+\|\lambda\|)^{s'} \big | \int_{i\fa^*}  (1+\|\lambda\|)^{\dim U-1-s'} \, \d \lambda.
\end{align*}
Taking $s'=\dim U +\dim A-1+s$, the assertion of the proposition follows with \eqref{eq:31.12.2019}. 
\end{proof}

\begin{rem}
Note that in the situation of the  proposition above, one can actually show by examining the critical set  of the square of $\delta_\beta(X)$  that   
\bq
\label{eq:13.10.19}
\int_D f_\eps (g^{-1}\beta g)\, d g \leq  C_{D,m} \,  \eps
\eq
for any  $0<\eps \leq \eps(D,m)$,  yielding a better power in $\eps$. Nevertheless, both $\eps(D,m)$ and the constant $0 < C_{D,m}$ cannot be specified in their dependence of $m$ with this method, which is essential for the obtention of  Sato-Tate equidistribution results. But since the proof of  \eqref{eq:13.10.19} does not require the theory of orbital integrals and has an interest in its own, we include it below. In the case $\beta \notin K$, the proof of \eqref{eq:13.10.19} is identical to the one in Proposition \ref{lem:20190605}.  Let us therefore suppose that $\beta \in K\cap H(\Q)\cap \M\big (N,\frac 1m \Z \big )$ for some $m \in \N$ and $\beta \notin \widetilde K$. Then $\inf \delta_\beta=0$, and
\bqn
S_\beta=\mklm{X \in \p \mid \delta_\beta(X)=0}=\mklm{X\in \p \mid \exp(-X) \cdot \beta \cdot \exp X \in K }
\eqn
is a submanifold of lower dimension by Lemma \ref{lem:22.5.2019}. In fact,  $S_\beta\subset \p$ coincides with the critical set of $\Delta_\beta(X):=\delta_\beta(X)^2/2$, see  \cite[p. 279]{helgason78} and \cite[Proposition 5.7]{DKV1}. Furthermore,  $S_\beta$ is clean as critical set of $\Delta_\beta$, compare \cite[p. 64]{DKV1}, and setting  
$$\mathcal{U}_\beta(\eps):=\mklm{ X \in D_\p \mid \delta_\beta(X) < \eps }$$
we conclude  by definition of $f_\eps$ that
\bqn
\int_{D_\p} f_\eps (\exp(-X) \cdot \beta \cdot \exp X)\, d X=\vol (D_\p\cap \mathcal{U}_\beta(\eps)).
\eqn
 The fact that  $S_\beta$ is clean as critical set of $\Delta_\beta$ means that the transversal Hessian $\mathrm{Hess}^\perp \Delta_\beta$ of $\Delta_\beta$ is non-degenerate, which together with the fact that $\Delta_\beta$ takes its minimum at $S_\beta$ implies that $\mathrm{Hess}^\perp \Delta_\beta$ has strictly positive eigenvalues. By compactness we can therefore choose an  $\eps( D_\p,m)>0$  independent of $\beta$ such that
\begin{itemize}
\item there exist finitely many charts $\mklm{(\kappa_\iota,\O_\iota)}_{\iota \in I}$ which cover $D_\p \cap \mathcal{U}_\beta\big (\eps(N, D_\p,m)\big )$ such that for each $\iota \in I$
\bqn 
\kappa_\iota^{-1}(x,y) \in S_\beta \quad \Longleftrightarrow  \quad y=0,
\eqn
where $\kappa_\iota: \p \supset \O_\iota \in X \mapsto (x,y) \in \R^{\dim S_\beta} \times \R^{\dim \p -\dim S_\beta}$ are the corresponding  local coordinates;
\item for each $\iota \in I$ and $x$, the function $y \mapsto \Delta_\beta \circ \kappa_\iota^{-1} (x,y)$ has a non-degenerate critical point $y=0$;
\item for each $\iota \in I$ and  $(x,y) \in \kappa_\iota(\O_\iota)$,  the  real symmetric matrix
\bqn 
\H_\iota(x,y):= \left ( \frac{\gd^2}{\gd y_i \gd y_j} (\Delta_\beta \circ \kappa_\iota^{-1}) (x,y)\right )_{i,j}
\eqn
has only strictly positive eigenvalues, $\H_\iota(x,0)$ being equal to $\mathrm{Hess}^\perp \Delta_\beta(\kappa_\iota^{-1} (x,0))$. 
\end{itemize}
Now,  let $X=\kappa_\iota^{-1}(x,y)$ for some $\iota$ and $(x,y) \in \kappa_\iota(\O_\iota)$. Since $\Delta_\beta \circ \kappa_\iota^{-1} (x,\cdot )$ vanishes in second order at $y=0$,  Taylor expansion in transversal direction at $y=0$ yields  
\bq
\label{eq:28.7.2019}
\Delta_\beta(X) = \frac 12 \sum_{i,j}  \frac{\gd^2}{\gd y_i \gd y_j} (\Delta_\beta \circ \kappa_\iota^{-1}) (x,y_0) \,  y_i \, y_j= \frac 12 \eklm{y, \H_\iota(x,y_0) y}
\eq
for some $y_0$ lying on the line segment $[0,y]$ joining $0$ and $y$. By the theorem of Courant-Fischer, for $y\neq 0$ one has
\bqn 
\lambda_\text{min} \leq \frac {\eklm{ y, \H_\iota(x,y_0) y}}{\eklm{y,y}} \leq \lambda_\text{max},
\eqn
$\lambda_\text{min}$ and $\lambda_\text{max}$ denoting the minimal and maximal eigenvalue of $\H_\iota(x,y_0)$, so with \eqref{eq:28.7.2019} we infer for any $0<\eps < \eps(D_\p,m )$ that 
\bqn 
\delta_\beta(X) < \eps \quad \Longrightarrow \quad \norm{y}  \leq \lambda_\text{min}^{-1} \eps. 
\eqn
Thus, by compactness there is a constant $C_{D_\p,m }>0$   such that 
\bqn 
D_\p\cap \mathcal{U}_\beta(\eps) \subset \bigcup_{\iota \in I} \mklm{X=\kappa_\iota^{-1}(x,y) \in \O_\iota \mid \norm{y} \leq C_{D_\p,m}   \, \eps}.
\eqn
 Since $S_\beta$ has at least codimension $1$ we conclude that 
\bqn
\vol (D_\p\cap \mathcal{U}_\beta(\eps)) = O_{D_\p,m }( \eps)
\eqn
for any $0<\eps < \eps(D_\p,m )$, finishing the proof of \eqref{eq:13.10.19}.
\end{rem}

We can now state the second main result of this paper.  As before, let $\{ \phi_j \}_{j\in\N}$ be an orthonormal basis of $\L^2(H(\Q)\bsl H(\bA)/K_0)$  such that each $\phi_j$ is an eigenfunction of $\Delta$  included in a single space $V_\pi$. In particular, each $\phi_j$ is a simultaneous eigenfunction of $\Delta$ and the Hecke operators $T_{K_0\alpha K_0}$, $\alpha\in H(\bA_\fin^{S_0})$.
\begin{thm}[\bf Equivariant distribution of Hecke eigenvalues]\label{thm:equiv} 
Let  $H_1$ be a simple\footnote{This means that $H_1$ has no nontrivial connected normal subgroups.} connected algebraic group over a number field $F$, and set $H:=\mathrm{Res}_{F/\Q}(H_1)$.\footnote{Here $\mathrm{Res}_{F/\Q}$ means the restriction of scalars from $F$ to $\Q$, see \cite[Section 2.1.2]{PR}.}  With the notation of the beginning of Section \ref{Equivariant case}, write $d:=\dim H(\R)$ and let $N \geq d-\dim K+1$ be a  sufficiently large integer such that  one has an embedding $H(\R)\subset \SL(N,\R)$. Suppose that $\sigma \in \widehat K$ is trivial on $Z$.
Then,   there exists a constant $0<c<N^2+2N$ such that for any finite set $S$ of primes in the complement of  $S_0$, any $\alpha\in H(\Q_S)$ with $\|\alpha\|_S\leq \kappa$, and any $0<s<1$
\begin{align*}
\sum_{\stackrel{\mu_j\leq \mu,}{ \phi_j \in \L^2_\sigma(H(\Q)\bsl H(\bA)/K_0) }} \lambda_j(\alpha) &=  \frac{n_Z(\alpha) \, d_\sigma \,  \vol(M/K) \, \varpi_{d-\dim K} }{(2\pi)^{d-\dim K} }  \mu^{d-\dim K}  \\ &+  O(\mu^{d-\dim K-\frac{d-\dim K-1}{d-\dim K+1}(1-s)}\, p_S^{c\kappa}\, s^{-1}) 
\end{align*}
for sufficiently large $\mu$,\footnote{More precisely, for $ \mu \gg N^{(d-\dim K+1)/(d-\dim K-1)}$.} where $n_Z(\alpha):=|Z_H(\Q) \cap (K \cdot K_0\alpha K_0)|$ and $\vol(M/K)$ denotes the orbifold volume of $M/K$.
\end{thm}
\begin{proof}

To begin, note that $G$ might have a compact simple factor. Nevertheless, the conclusions of Lemmata \ref{lem:22.5.2019} and  \ref{lem:zero} are still true if $\beta\in H(\Q)$. Indeed, let $F_v$ denote the completion of $F$ at a place $v$ of $F$. Since $G$ is isomorphic to $\prod_{v|\infty} H_1(F_v)$, where $v$ moves over infinite places of $F$, there exists an infinite place $w$ of $F$ such that $H_1(F_w)$ is not compact. In addition,   $H_1(F_w)$ is simple by assumption. Therefore, if  $\beta\in H(\Q)$, it  is sufficient to apply Lemma \ref{lem:22.5.2019} to $H_1(F_w)$, $H(\Q)=H_1(F)$ being directly embedded into $\prod_{v|\infty} H_1(F_v)$. In view of Lemma \ref{lem:21.04.2018} we are therefore left with the task of  deriving suitable upper bounds for  the sums of orbital integrals
\bqn
\sum_{\beta\in \Gamma_l\bsl \Gamma_l\alpha_{l,l,m}\Gamma_l} \int_{\Gamma_l\bsl G} ( f_\eps(g^{-1}\beta g)-f_0(g^{-1}\beta g) )  \d g,
\eqn
where $l$ ranges from $1$ to $c_H$, $m$  from $1$ to $c_{l,l}$. For this, choose a connected compact domain $D$ in $G$ including a fundamental domain of $\Gamma_l\bsl G$, and let $\beta_1,\dots , \beta_r\in \Gamma_l\bsl \Gamma_l\alpha_{l,l,m}\Gamma_l$ be a set of representative elements, that is, $\Gamma_l\alpha_{l,l,m}\Gamma_l=\Gamma_l\beta_1\sqcup\cdots\sqcup \Gamma_l\beta_r$ where $r=|\Gamma_l\bsl \Gamma_l\alpha_{l,l,m}\Gamma_l|$. Now, observe that
\bqn 
f_\eps(g^{-1}\beta g)-f_0(g^{-1}\beta g)=0 \qquad \Longleftrightarrow \qquad \begin{cases} \dist_G(K, g^{-1}\beta gK)> \eps \quad \text{or} \\ \dist_G(K, g^{-1}\beta gK)=0. \end{cases}
\eqn
The condition $ \dist_G(K, g^{-1}\beta gK)=0$ is equivalent to $N(\beta,K)$ not being empty, in which case Lemma \ref{lem:22.5.2019} asserts that $N(\beta,K)$ has full measure iff $\beta \in C(G)$. 
In addition,  by the argument above, $\beta\in\widetilde K\cap H_1(F)$ implies $\beta\in C(G)$.
Therefore, taking  into account  Proposition \ref{lem:20190605},  the  integrals in question can be estimated according to
\begin{align*}
\sum_{\beta\in \Gamma_l\bsl \Gamma_l\alpha_{l,l,m}\Gamma_l} & \int_{\Gamma_l\bsl G} ( f_\eps(g^{-1}\beta g)-f_0(g^{-1}\beta g) )  \d g  \ll \sum_{\stackrel{j=1,}{\beta_j\not \in \widetilde K}}^{r} \int_D f_\eps(g^{-1}\beta_j g)   \d g \ll  \eps^{1-s} \, p_S^{(N+c_3)\kappa} s^{-1}
\end{align*}
uniformly in $l$ and $m$, provided that
 \bq\label{eq:30.1.2020}
 0<\eps \ll \log (1+N^{-1} p_S^{-2\kappa} c_1^{-2}),
 \eq
 where  $c_j$ denote  the same constants $c_j$ than in the proof of Theorem \ref{thm:main}. Here we put $m=c_1p_S^\kappa$ in Proposition \ref{lem:20190605}, and took into account that   $r$ is bounded by $p_S^{c_3\kappa}$ up to a constant. 
Integrating over $x$ and $\mu$ we now infer from  Lemmata  \ref{lem:zero} and \ref{lem:21.04.2018} that
\begin{align*}
\int_{-\infty}^\mu \int_{\Gamma_l \bsl G}  \Big[ K_{T_\alpha\circ \widetilde s_t\circ \Pi_\sigma} (x,x) &-  n_Z(\alpha)  K_{\widetilde s_t \circ \Pi_\sigma} (x,x) \Big ] \d x \d t \\
& = O\Big( \mu^{d-\dim K} \eps^{1-s} \, p_S^{(N+c_3)\kappa}s^{-1}  +     \mu^{(d-\dim K + 1)/2} \eps^{-(d-\dim K -1)/2} p_S^{c_3\kappa} \Big ),
\end{align*}
where we took into account \eqref{eq:9.6.19}. 
Putting
\bq
\label{eq:30.1.2020a}
\eps=\mu^{-(d-\dim K-1)/(d-\dim K+1)}p_S^{-2N\kappa/(d-\dim K +1)}
\eq
 the assertion of the  theorem now follows from \eqref{eq:12.4.2018} by the same arguments than those in the proof of Theorem \ref{thm:main}. Notice hereby  that for  $N\geq d-\dim K+1$ the choice \eqref{eq:30.1.2020a} in particular fulfills the requirement \eqref{eq:30.1.2020}
 for $\mu \gg N^{(d-\dim K+1)/(d-\dim K-1)}$. Furthermore, 
$$
\int_M [\pi_{\sigma|K_x}:\1] \vol [(\Omega \cap S^\ast_x(M))/ K]\, \d x=[\pi_{\sigma|K_\text{prin}}:\1] \int_{M_{(K_\text{prin})}}  \vol [(\Omega \cap S_x^\ast(M))/ K] \, \d x,
$$
where $(K_\text{prin})$ denotes the principal isotropy type of the $K$-action on $M$, $K_\text{prin}\subset K$ being a closed subgroup, and $M_{(K_\text{prin})}$  the corresponding stratum  of $M$, the latter being dense.  But since  there are only finitely many torsion points on a fundamental domain of an arithmetic quotient  \cite[Theorem 4.15]{PR}  we have $K_\text{prin}=\mklm{1}$, and consequently $[\pi_{\sigma|K_\text{prin}}:\1]=1$.
In addition, by singular cotangent bundle reduction \cite[Remark 3.4]{kuester-ramacher17} we have 
$
\Omega /K \simeq T^\ast (M/K)
$ 
as orbifolds,  so  that as in \eqref{eq:26.9.2019} one deduces 
\begin{align*}
 \int_{M} \vol [(\Omega \cap S_x^\ast(M))/ K] \, \d x&=  (d-\dim K) \int_{M/K} \vol [B_{x\cdot K}^\ast(M/K) ] \, \d (x\cdot K) \\ &= (d-\dim K) \, \varpi_{d-\dim K} \, \vol(M/K).
\end{align*}
\end{proof}

Next, we shall introduce for each $\sigma$ in $\widehat{K}$ a  family  of $\sigma$-isotypic automorphic representations  of $H$, and recall for this purpose   the notations $\pi=\pi_\infty\otimes\pi_\fin$, $V_\pi=V_{\pi_\infty}\otimes V_{\pi_\fin}$, $V_{\pi_\infty}^{\leq \mu}$, and $V_{\pi_\fin}^{K_0}$ introduced in Section \ref{Non-equivariant case}  for each  $\pi\in\widehat{H(\bA)}$. The Peter-Weyl theorem implies the decompositions
\[
V_{\pi_\infty}^{\leq \mu}=\bigoplus_{\sigma\in \widehat{K}}V_{\pi_\infty,\sigma}^{\leq \mu},
\]
where $V_{\pi_\infty,\sigma}^{\leq \mu}$ denotes the $\sigma$-isotypic component in $V_{\pi_\infty}^{\leq \mu}$. 
Let now $\F_\sigma:=\F_\sigma(\mu)$ be the finite multi-set consisting of those automorphic representations $\pi\in\widehat{H(\bA)}$ that satisfy 
\[
a_{\F_\sigma}(\pi):=m_\pi \, \dim V^{\leq \mu}_{\pi_\infty,\sigma} \, \dim V_{\pi_\fin}^{K_0} >0,
\]
where  each such $\pi$ appears in $\F_\sigma$ with multiplicity $a_{\F_\sigma}(\pi)$.
Notice that $\dim V_{\pi_\infty,\sigma}^{\leq \mu}$ means the multiplicity of $\sigma$ in $\pi_\infty|_K$ if $V_{\pi_\infty,\sigma}^{\leq \mu}$ is not empty.
We also define a counting measure $\widehat{m}^\mathrm{count}_{\mu,\sigma,S}$ on $H(\Q_S)^{\wedge,\mathrm{ur}}$ for the $S$-component of $\F_\sigma$ by setting  
\bqn
\widehat{m}^\mathrm{count}_{\mu,\sigma,S}:=\frac{1}{|\F_\sigma|}\sum_{\pi\in \F_\sigma} \delta_{\pi_S}.
\eqn

The following results are direct consequences of  Theorem \ref{thm:equiv} and can be proved by the same arguments used in Section \ref{Non-equivariant case} to prove their non-equivariant versions.

\begin{cor}[\bf Equivariant asymptotic trace formula]\label{cor:equiv}
Choose a $K$-type $\sigma\in \widehat{K}$, and  for simplicity  suppose that\footnote{This ensures that $Z=Z_H(\Q)\cap K$. Otherwise,  a central character has to be fixed.}  $Z_H(\Q)\cap K$ is contained in $K_0$.
Then, there exists a  constant $c'>0$ such that  for each finite set $S$ of primes outside $S_0$, each $f_S\in \mathcal{H}^\mathrm{ur}_\kappa(H(\Q_{S}))$ with $|f_S|\leq 1$, and each $0<s<1$
\[
\sum_{\pi\in \F_\sigma(\mu)} \Tr\pi_S(f_S)=f_S(1)\cdot  \frac{ |Z| \, d_\sigma \, \vol(M/K)  \, \varpi_{d-\dim K} }{(2\pi)^{d-\dim K} } \mu^{d-\dim K} +  O(\mu^{d-\dim K-\frac{d-\dim K-1}{d-\dim K+1}(1-s)}\, p_S^{c'\kappa}s^{-1})
\]
for sufficiently large $\mu$.\footnote{In fact, taking $ \mu \gg N^{(d-\dim K+1)/(d-\dim K-1)}$ is sufficient.}
\end{cor}
\qed
\begin{cor}[\bf Equivariant Plancherel theorem]\label{cor:equiv.planch}
For any $f_S\in \mathcal{H}^\mathrm{ur}(H(\Q_{S}))$, 
\[
\lim_{\mu\to\infty} \widehat{m}^\mathrm{count}_{\mu,\sigma,S}(\widehat {f_S}) = \widehat{m}_S^\mathrm{Pl,ur}(\widehat {f_S}).
\]
\end{cor}
\qed 
\begin{cor}[\bf Equivariant  Sato-Tate equidistribution theorem]\label{cor:equiv.ST}
Fix  $\theta\in\mathscr{C}(\Gamma_1)$, and let $\widehat{f}$ be a continuous function on $\widehat{T}_{c,\theta}/ \Omega_{c,\theta}$, which can be extended to a continuous function $\widehat{f}_p$ on $G_p^{\wedge,\mathrm{ur,temp}}$ for any $p\in \mathcal{V}(\theta)$ by \eqref{eq:ST(5.2)}.
If one now  chooses a sequence $\{(p_k,\mu_k)\}_{k\geq 1}$ in $\mathcal{V}(\theta)\times\R_{>0}$ such that $p_k\to \infty$ and $p_k^l / \mu_k \to 0$ as $k\to \infty$ for any integer $l\geq 1$, then 
\[
\lim_{k\to\infty} \widehat{m}^\mathrm{count}_{\mu_k,\sigma,p_k}(\widehat {f}_{p_k}) = \widehat{m}^\mathrm{ST}(\widehat {f}).
\]
\end{cor}
\qed

\section{Examples}\label{sec:example}
To conclude,  we shall specify some concrete situations to which our results apply.  As Section \ref{sec:nonadelic}, this section intends to make our results comprehensible to a wider audience,  illustrating their scope in some concrete situations.

\subsection{Algebraic groups arising from division algebras}
Let $D$ denote a central division algebra of index $n$ over a number field $F$.
The algebraic group 
$$H:=\mathrm{Res}_{F/\Q}\SL(1,D)$$ 
over $\Q$ is semisimple and simply connected, and a cocompact lattice $\Gamma$ of $G:=H(\R)$ is given by $\Gamma:=H(\Q)\cap K_0$ for each open compact subgroup $K_0$ in $H(\bA_\fin)$.
In this case,  $H(\Q)\bsl H(\bA)/K_0\cong \Gamma\bsl G,$
and our results apply. 
The real Lie group $G$ can be expressed as $G=\prod_{v|\infty}H(F_v)$, where $v$ moves over infinite places of $F$ and $F_v$ denotes the completion of $F$ at an arbitrary place $v$ of $F$, and for each infinite place $v$, the real group $H(F_v)=\SL(1,D\otimes F_v)$ is isomorphic to 
$
\SL(n,\C), \,  \SL(n,\R), \text{ or } \SL(n/2,\bH),
$
where $\bH$ denotes Hamilton's quaternion field and $n$ is even in $ \SL(n/2,\bH)$.  One could also consider  a quadratic extension $E$  of a number field $F$, together with a central division algebra $D$ over $E$ with  $E/F$-involution $\iota$.
Such division algebras have been classified in  \cite[Chapter 10]{Scharlau}. For a  fixed $\iota$ one can then take 
$$H:=\mathrm{Res}_{F/\Q}\SU(1,D)$$    as algebraic group, which is semisimple, simply connected, and connected, and  $\Gamma:=H(\Q)\cap K_0$ as cocompact discrete subgroup of $G:=H(\R)$.

\subsection{Special orthogonal  and quaternion special unitary groups} Our next class of examples consists of  special orthogonal and quaternion unitary groups constructed with   Borel's method \cite{borel63}. Choose a positive rational number $d\in\Q_{>0}^\times\setminus (\Q^\times)^2$ and consider the real quadratic field $F:=\Q(\sqrt{d})\subset \R$.
Set
\[
J_{p,q}:=\mathrm{diag}(\underbrace{1,\dots,1}_p,\underbrace{-\sqrt{d},\dots,-\sqrt{d}}_q).
\]
Then  $\SO(J_{p,q}):=\{g\in \SL(n) \mid gJ_{p,q}{}^t\!g=J_{p,q} \}$ constitutes a semisimple algebraic group over $F$, and we  set $H:=\mathrm{Res}_{F/\Q}\SO(J_{p,q})$. In this case, 
\[
G:=H(\R)\cong G_1\times G_2, \quad G_1:=\SO(p,q), \quad G_2:=\SO(p+q), 
\]
where $\SO(p,q)$ denotes the special orthogonal group of signature $(p,q)$ over $\R$, and it is well known that $\Gamma:=H(\Q)\cap K_0$ is cocompact.
Let $\Gamma_l$ denote an arithmetic lattice in $G$ defined as in \eqref{eq:20190622}. If the $K$-type $\sigma \in \widehat K$ is trivial on $G_2$, then $L^2_\sigma(K_0\bsl H(\bA)/K)$ can be identified with a sum $\bigoplus_{l=1}^{c_H} L^2_\sigma(\Gamma_{l,1}\bsl G_1)$, where $\Gamma_{l,1}$ denotes the projection of $\Gamma_l$ into $G_1$. 
In this case, our results imply asymptotics for the single orthogonal group $G_1$. Next, let $\sigma$ denote the non-trivial element of the Galois group $\mathrm{Gal}(F/\Q)$, and recall that there exist  quaternion division algebras $D_1$ and $D_2$  over $F$ such that $D_1\otimes_F \R\cong \M(2,\R)$, $D_1\otimes_{F^\sigma}\R\cong \bH$ and $D_2\otimes_F\R\cong D_2\otimes_{F^\sigma}\R\cong \bH$.
Introducting a conjugation map $x\mapsto \overline{x}$ on $D_j$, we can define over $F$ the semisimple algebraic groups $\SU(n,D_1):=\{g\in\SL(n,D_1) \mid g{}^t\!\overline{g}=I_n \}$ and  $\SU(J_{p,q},D_2):=\{g\in\SL(n,D_2) \mid gJ_{p,q}{}^t\!\overline{g}=J_{p,q} \}$, and we set
$$H:=\mathrm{Res}_{F/\Q}\SU(n,D_1) \, \text{ or } \, \mathrm{Res}_{F/\Q}\SU(J_{p,q},D_2).
$$
In these cases, $\Gamma=K_0\cap H(\Q)$ is cocompact. In the first case,  the real group $G:=H(\R)$ is isomorphic to $\Sp(2n)\times \SU(n,\bH)$ where $\Sp(2n)$ denotes the split symplectic group of rank $n$ over $\R$ and $\SU(n,\bH):=\{ g\in\SL(n,\bH) \mid g{}^t\!\overline{g}=I_n\}$, while in the second case  the real group is isomorphic to $\SU(p,q,\bH)\times \SU(n,\bH)$,  where $\SU(p,q,\bH)$ denotes the quaternion special unitary group of signature $(p,q)$.



\providecommand{\bysame}{\leavevmode\hbox to3em{\hrulefill}\thinspace}
\providecommand{\MR}{\relax\ifhmode\unskip\space\fi MR }
\providecommand{\MRhref}[2]{%
  \href{http://www.ams.org/mathscinet-getitem?mr=#1}{#2}
}
\providecommand{\href}[2]{#2}


\end{document}